%% file: main.tex
\documentclass[a4paper,11pt, twoside]{article}
\usepackage{a4}
\usepackage{geometry} 
\usepackage{amsmath}
\usepackage{amssymb}
\usepackage{amsthm}
\usepackage{graphicx}
\usepackage{caption,subcaption}
\usepackage{multirow}
\usepackage{xstring}
\usepackage[utf8]{inputenc}
\usepackage{hyperref}
\usepackage{amstext,amsbsy,amsopn,eucal,enumerate}
\usepackage{tikz}
\usepackage{pgfplots}
\usepackage{tabularx}

\usepackage{soul}
\usepackage{algorithm} %ctan.org\pkg\algorithms
\usepackage{algpseudocode}

\def\R{\mathbb{R}}

\newcommand{\expn}{\operatorname{e}}

\newcommand{\diag}{\operatorname{diag}}

\newcommand{\orth}{\operatorname{orth}}
\newcommand{\vect}{\operatorname{vec}}

\newcommand{\beq}{\begin{equation}}
\newcommand{\eeq}{\end{equation}}
\newcommand {\mat}      [1] {\left[\begin{array}{#1}}
\newcommand {\rix}          {\end{array}\right]}
\newcommand {\smat}      [1] {\left[\begin{smallmatrix}{#1}}
\newcommand {\srix}          {\end{smallmatrix}\right]}
\newcommand {\s}      [1] {\begin{smallmatrix}{#1}}
\newcommand {\se}          {\end{smallmatrix}}
\newcommand{\trace}{\operatorname{tr}}

\newcommand{\hA}{\ensuremath{\hat{A}}}

\newtheorem{defn}{Definition}[section]
\newtheorem{remark}{Remark}
\newtheorem{assumption}{Assumption}

\newtheorem{lem}[defn]{Lemma}
\newtheorem{prop}[defn]{Proposition} %Proposition entspricht Satz

\newtheorem{thm}[defn]{Theorem}

\usetikzlibrary{external}
\def\addlegendimage{\csname pgfplots@addlegendimage\endcsname}

\newlength\fheight
\newlength\fwidth

\title{Optimization based model order reduction for stochastic systems}

\author{Martin Redmann\thanks{Martin Luther University Halle-Wittenberg, Institute of Mathematics, Theodor-Lieser-Str. 5, 06120 Halle (Saale), Germany, Email: {\tt 
martin.redmann@mathematik.uni-halle.de}.}
\and Melina A. Freitag\thanks{Institut f\"{u}r Mathematik, Universit\"{a}t Potsdam, 
Campus Golm, Karl-Liebknecht-Str. 24-25, 14476 Potsdam, Germany, Email: {\tt melina.freitag@uni-potsdam.de}}
}

\begin{document}

\maketitle

\begin{abstract}
\input{abstract.tex}
\end{abstract}


%%%%%%%%%%%%%%%%%%%%%%

\input{introduction.tex}

\input{mult.tex}

\input{add.tex}

\input{numerics.tex}

\input{conclusions.tex}

\appendix

\input{appendix.tex}

\bibliographystyle{plain}

\end{document}

%% file: abstract.tex
In this paper, we bring together the worlds of model order reduction for stochastic linear systems and $\mathcal H_2$-optimal model order reduction for deterministic systems. In particular, we supplement and complete the theory of error bounds for 
model order reduction of stochastic differential equations. With these error bounds, we establish a link between the output error for stochastic systems (with additive and multiplicative noise) and modified versions of the $\mathcal H_2$-norm for both linear and bilinear deterministic systems. When deriving the respective optimality conditions for minimizing the error bounds, we see that model order reduction techniques related to iterative rational Krylov algorithms (IRKA) are very natural and effective methods for reducing the dimension of large-scale stochastic systems with additive and/or multiplicative noise. We apply modified versions of (linear and bilinear) IRKA to stochastic linear systems and show their efficiency in numerical experiments. 

%% file: introduction.tex
\section{Introduction}

We consider the following linear stochastic systems \begin{subequations}\label{original_system}
\begin{align}\label{stochstate}
 dx(t) &= [Ax(t) +  B_1 u(t)]dt + f(x(t), dM(t)),\quad x(0) = x_0,\\ \label{stochout}
y(t) &= Cx(t),\quad t\geq 0,
\end{align}
            \end{subequations}
and the function $f$ represents  either additive or multiplicative noise, i.e., \begin{align*}
f(x(t), dM(t)) = \begin{cases}
  B_2 dM(t),  & \text{additive case }\\
  \sum_{i=1}^{m_2} N_i x(t-) dM_i(t), & \text{multiplicative case},
\end{cases}
\end{align*}
where $x(t-):=\lim_{s\uparrow t} x(s)$. Above, we assume that $A, N_i \in \R^{n\times n}$, $B_1 \in \R^{n\times m_1}$, $B_2 \in \R^{n\times m_2}$, and $C \in \R^{p\times n}$ are constant matrices. 
The vectors $x(t)\in\R^n$, $u(t)\in\R^{m_1}$ and $y(t)\in\R^p$ are called state, control input and output vector, respectively. Moreover, let 
$M=\left(M_1, \ldots, M_{m_2}\right)^T$ be an $\mathbb R^{m_2}$-valued square integrable L\'evy process with mean zero and covariance matrix $K=(k_{ij})_{i, j = 1, \ldots, m_2}$, i.e.,
$\mathbb E[M(t)M^T(t)]= K t$ for $t\geq 0$. Such a matrix exists, see, e.g., \cite{zabczyk}.

$M$ and all stochastic process appearing in this paper are defined on a filtered probability space
$\left(\Omega, \mathcal F, (\mathcal F_t)_{t\geq 0}, \mathbb P\right)$\footnote{$(\mathcal F_t)_{t\geq 0}$ is right continuous and complete.}. In addition, $M$ is $(\mathcal F_t)_{t\geq 0}$-adapted and its 
increments $M(t+h)-M(t)$ are independent of $\mathcal F_t$ for $t, h\geq 0$. Throughout this paper, we assume that $u$ is an $(\mathcal F_t)_{t\geq 0}$-adapted control that is square integrable meaning that 
\begin{align*}
\left\| u\right\|^2_{L^2_T} := \mathbb E\int_0^T \left\| u(s)\right\|^2 ds < \infty
\end{align*}
for all $T>0$.

In recent years, model order reduction (MOR) techniques such as balanced truncation (BT) and singular perturbation approximation (SPA), methods 
well-known and well-understood for deterministic systems \cite{antoulas,spa,moore} have been extended to stochastic systems of the form \eqref{original_system}, see, for example \cite{bennerdamm,redmannbenner,redSPA,redmannfreitag}.
In this paper we discuss optimization based model order reduction techniques for stochastic systems, which will lead naturally to iterative rational Krylov algorithms (IRKA), methods well-known for deterministic systems. 

IRKA was introduced in \cite{linIRKA} (for systems \eqref{original_system} with $f\equiv 0$) and relies on finding a suitable bound ($\mathcal H_2$-error) for the output error of two systems with the same structure but, as
in the context of MOR, one is usually large-scale and the other one is of small order. Subsequently, first order optimality conditions for this $\mathcal H_2$-bound with respect to the reduced order model (ROM) coefficients
were derived. 
These optimality conditions can be based on system Gramians \cite{hyland, Wilson} or they can be equivalently formulated as interpolatory conditions \cite{linIRKA,meierlun} associated to transfer functions of the systems. It
was shown in \cite{linIRKA} that IRKA fits these conditions. This $\mathcal H_2$-optimal scheme was extended in the linear deterministic setting to  minimizing systems errors in weighted norms \cite{linWeightedH2,NearoptfreqMOR,halevi}.

An extension of IRKA to bilinear systems, which relies on the Gramian based optimality conditions shown in \cite{morZhaL02} was given in \cite{breiten_benner}. The respective interpolatory optimality conditions in the bilinear 
case were proved in \cite{flagggug}. However, in contrast to the linear case, bilinear IRKA in \cite{breiten_benner} was developed without knowing about the link between the bilinear $\mathcal H_2$-distance and the output error of 
two bilinear systems. Later this gap was closed in \cite{redmann_h2_bil} showing that the bilinear $\mathcal H_2$-error bounds the output error if one involves the exponential of the control energy.

In order to establish IRKA for stochastic systems \eqref{original_system} as an alternative to balancing related MOR, we develop a theory as follows. We prove an output error between two stochastic systems with multiplicative
noise and derive the respective first order conditions for optimality in Section \ref{sec:mult}. The bound in the stochastic case \eqref{eq:boundstoch} covers the $\mathcal H_2$-error of two bilinear system as special cases,
the same is true for the optimality conditions which generalize the ones in \cite{morZhaL02}. However, in the stochastic case, in contrast to the bilinear case, the bound does not include the additional factor of the exponential
of the control energy. Hence the bound is expected to be much tighter for stochastic system. Based on the optimality conditions for \eqref{eq:boundstoch} we propose a modified version of bilinear IRKA. Based on the structure of
the bound in \eqref{eq:boundstoch} modified bilinear IRKA appears to be an even more natural method to reduce stochastic systems rather than bilinear systems. 

For the case of additive noise, which we consider in Section \ref{sec:add}, the first order optimality conditions are merely a special case of the ones for multiplicative noise. As an additional feature we introduce a splitting
approach for systems with additive noise in this section, where we split the linear system into two subsystems; one which includes the deterministic part and one the additive noise. We reduce each subsystem independently, which
allows for additional flexibility, in case that one of the systems is easier to reduce. Moreover, we consider a one step approach which reduced the deterministic and the noisy part simultaneously. Again, error bounds are provided
which naturally lead to (weighted) versions of linear IRKA for each subsystem for the reduction process. 

The final Section \ref{sec:numerics} contains numerical experiments for systems with both multiplicative and additive noise in order to support our theoretical results.

%% file: mult.tex
\section{Systems with multiplicative noise}
\label{sec:mult}

We study the multiplicative case first, in which \eqref{original_system} becomes 
\begin{subequations}\label{original_system_mul}
\begin{align}\label{stochstate_mul}
 dx(t) &= [Ax(t) +  B_1 u(t)]dt + \sum_{i=1}^{m_2} N_i x(t-) dM_i(t),\quad x(0) = x_0,\\ \label{stochout_mul}
y(t) &= Cx(t),\quad t\geq 0.
\end{align}
\end{subequations}
Now, the goal is to find a measure for the distance between \eqref{original_system_mul} and 
a second system having the same structure but potentially a much smaller dimension. It is given by 
\begin{subequations}\label{red_system_mul}
\begin{align}\label{red_stochstate_mul}
 d\hat x(t) &= [\hat A\hat x(t) +  {\hat B}_1 u(t)]dt + \sum_{i=1}^{m_2} {\hat N}_i \hat x(t-) dM_i(t),\quad \hat x(0) = {\hat x}_0,\\ \label{red_stochout_mul}
\hat y(t) &= \hat C\hat x(t),\quad t\geq 0,
\end{align}
\end{subequations}
where $\hat x(t)\in\mathbb R^r$, with $r\ll n$ and $\hat A$, ${\hat B}_1$, $\hat C$, ${\hat N}_i$, $i=1,\ldots,m_2$ of appropriate dimension. 
In order to find a distance between the above systems, a stability assumption and the fundamental solution to both systems are needed. 
Some results (in particular the one on the optimality for the error between two systems) obtained for systems with multiplicative noise can be transferred to the case with additive noise as we will see later. 

\subsection{Fundamental solutions and stability}

We introduce the fundamental solution $\Phi$ to \eqref{stochstate_mul}. It is defined as the $\R^{n\times n}$-valued solution to 
\begin{align}\label{funddef}
 \Phi(t, s)=I+\int_s^t A \Phi(\tau, s) d\tau+\sum_{i=1}^{m_2} \int_s^t N_i \Phi(\tau-, s)dM_i(\tau), \quad t\geq s.
\end{align}
This operator maps the initial condition $x_0$ to the solution of the homogeneous state equation with initial time $s\geq 0$. We additionally define $\Phi(t) := \Phi(t, 0)$. 
Note, that $\Phi$ also includes the fundamental solution of the additive noise scenario which we will use later. It is obtained by setting 
$N_1=\ldots=N_{m_2} = 0$ in \eqref{funddef}, so that $\Phi(t, s) = \expn^{A(t-s)}$. We make a stability assumption on the fundamental solution which we need to produce well-defined system norms (and distances). 

\begin{assumption}\label{assumption}
The fundamental solution $\Phi$ is mean square asymptotically stable, i.e., there is a constant $c>0$ such that 
\begin{align}\label{assumptionstab}
  \mathbb E\left\|\Phi(t)\right\|^2 \lesssim \expn^{-c t} \Leftrightarrow \lambda\left(I\otimes A + A\otimes I+ \sum_{i, j=1}^{m_2} N_i \otimes N_j k_{ij}\right)\subset \mathbb C_-,
\end{align}
where $\lambda(\cdot)$ denotes the spectrum of a matrix. We refer to \cite{redmannspa2} for the equivalence in \eqref{assumptionstab}, or to \cite{staboriginal} for the same result in case of standard Wiener noise.
\end{assumption}
Note that, with additive noise ($N_i=0,\forall i$ in \eqref{funddef}), condition \eqref{assumptionstab} simplifies to 
\begin{align}\label{stab_cond_add}
 \left\|\expn^{At}\right\|^2 \lesssim \expn^{-c t} \Leftrightarrow \lambda\left(A\right)\subset \mathbb C_-.
\end{align}
The fundamental solution is a vital tool to compute error bounds between two different stochastic systems. The key result to establish
these bounds is the following lemma which is a generalization of \cite[Proposition 4.4]{redmannbenner}.
\begin{lem}\label{lemdgl}
Let $\Phi$ be the fundamental solution of the stochastic differential equation with coefficients $A, N_i \in \R^{n\times n}$ defined in \eqref{funddef} and let $\hat \Phi$ be the one of the same system, where $A, N_i$ are 
replaced by $\hat A, \hat N_i \in \R^{r\times r}$. Moreover, suppose that $L$ and $\hat L$ are matrices of suitable dimension. Then,  
the $\mathbb R^{n\times r}$-valued function $\mathbb E \left[\Phi(t, s) L {\hat L}^T \hat{\Phi}^T(t, s)\right]$, $t\geq s$, satisfies \begin{align}\label{dglmixed}
\dot{X}(t)=X(t) {\hat A}^T+A X(t)+\sum_{i, j = 1}^{m_2} N_i X(t) {\hat N}_j^T\;k_{ij},\quad X(s)= L {\hat L}^T .
\end{align}
\end{lem}
\begin{proof}
See Appendix \ref{proof_prop_fund}.
\end{proof}
Lemma \ref{lemdgl} yields \begin{align}\label{equamixedew}
\mathbb E \left[\Phi(t, s)  L {\hat L}^T \hat{\Phi}^T(t, s)\right]=\mathbb E \left[\Phi(t-s)  L {\hat L}^T \hat{\Phi}^T(t-s)\right]
\end{align}
for all $t\geq s\geq 0$, since both expressions solve \eqref{dglmixed}.

\subsection{The stochastic analogue to $\mathcal H_2$-norms}\label{stochH2}

For deterministic linear systems (system \eqref{original_system} without noise), a transfer function $G$ can be interpreted as an input-output map in the frequency domain, i.e., $\tilde y = G \tilde u$, where 
$\tilde u$ and $\tilde y$ are the Laplace transforms of the input and the  output, receptively. A norm associated with this transfer function can subsequently be defined that provides 
a bound for the norm of the output. This can, e.g., be the $\mathcal H_2$-norm of $G$ which in the linear deterministic case is given by 
\begin{align*}
\left\| G\right\|^2_{\mathcal H_2} := \frac{1}{2\pi} \int_{-\infty}^\infty \left\| G(\mathrm i w) \right\|^2_F dw,                                                                                                               
\end{align*}
where $\left\| \cdot\right\|_F$ denotes the Frobenius norm and $\mathrm i$ the imaginary unit.
However, there are no transfer functions in the stochastic case but we can still define a norm that is analogue to the $\mathcal H_2$-norm. To do so, 
we use a connection known in the linear deterministic case. There, the $\mathcal H_2$-norm of a transfer function coincides with the $L^2$-norm of the impulse response of the system, that is,
\[
\left\| G\right\|^2_{\mathcal H_2} = \int_0^{\infty} \left\| C e^{A s} B_1 \right\|^2_F ds,
\]
see \cite{antoulas}. A generalized impulse response exists in the stochastic setting \eqref{original_system_mul} and is given by $H(t):= C \Phi(t) B_1$, where $\Phi$ is the fundamental solution 
defined in \eqref{funddef}. We introduce a space $\mathcal L^2(\mathcal W)$ of matrix-valued and $(\mathcal F_t)_{t\geq 0}$-adapted stochastic processes $Y$ of appropriate dimension with 
\begin{align*}
\left\| Y\right\|^2_{\mathcal L^2(\mathcal W)} :=  \mathbb E\int_0^\infty \left\| Y(s) \mathcal W\right\|^2_F ds <\infty ,
\end{align*}
and $\mathcal W$ is a regular $m_1\times m_1$ matrix that can be seen as a weight (in the simplest case the identity matrix) and will be specified later. Now, the stochastic analogue to the (weighted) $\mathcal H_2$-norm 
for system \eqref{original_system_mul} is 
\begin{align*}
 \left\| H\right\|_{\mathcal L^2(\mathcal W)} =  \left(\mathbb E\int_0^\infty \left\| C \Phi(s) B_1 \mathcal W\right\|^2_F ds\right)^{\frac{1}{2}},
\end{align*}
which is finite due to Assumption \ref{assumption}. Let $\hat\Phi$ denote the fundamental solution to the reduced system \eqref{red_system_mul} and $\hat H(t):= \hat C \hat \Phi(t) \hat B_1$ its 
impulse response. Then, we can specify the distance between $H$ and $\hat H$. A Gramian based representation is stated in the next theorem.
\begin{thm}\label{thm_stoch_h2_rep}
Let $H$ and $\hat H$ be the impulse responses of systems \eqref{original_system_mul} and \eqref{red_system_mul}, respectively. Moreover, suppose that 
Assumption \ref{assumption} holds for \eqref{original_system_mul} and \eqref{red_system_mul}. Then, we have \begin{align*}
 \left\| H - \hat H\right\|_{\mathcal L^2(\mathcal W)}^2 =  \left\| C \Phi B_1 - \hat C \hat \Phi \hat B_1\right\|_{\mathcal L^2(\mathcal W)}^2 = \trace(C P C^T) +   \trace(\hat C \hat P {\hat C}^T) - 2 \trace(C P_2 {\hat C}^T),                                               
                                                       \end{align*}
where the matrices $P, \hat P$ and $P_2$ are the solutions to \begin{align}\label{full_gram}
A P + P A^T+ \sum_{i, j=1}^{m_2} N_i P N_{j}^T k_{ij} &= -B_1 \mathcal W (B_1 \mathcal W)^T,\\ \label{red_gram}
\hat A\hat P + \hat P {\hat A}^T+\sum_{i, j=1}^{m_2} \hat N_{i} \hat P {\hat N}_{j}^T k_{ij} &= -\hat B_1 \mathcal W (\hat B_1 \mathcal W)^T\\ \label{mixed_gram}
A P_2 + P_2 {\hat A}^T+\sum_{i, j=1}^{m_2} N_{i} P_2 {\hat N}_{j}^T k_{ij} &= -B_1 \mathcal W (\hat B_1 \mathcal W)^T.
    \end{align}
\end{thm}
\begin{proof}
We have that
  \begin{align*}
 \left\| H - \hat H\right\|_{\mathcal L^2(\mathcal W)}^2 =\;& \mathbb E\int_0^\infty \left\| (C \Phi(s) B_1 - \hat C \hat \Phi(s) \hat B_1) \mathcal W\right\|^2_F ds\\
 =\;&\mathbb E\int_0^\infty \trace\left(C \Phi(s) (B_1\mathcal W) (B_1\mathcal W)^T\Phi^T(s) C^T\right) ds\\
 &+ \mathbb E\int_0^\infty \trace\left(\hat C \hat \Phi(s) (\hat B_1\mathcal W) (\hat B_1\mathcal W)^T\hat\Phi^T(s) \hat C^T\right) ds\\
 &- 2 \mathbb E\int_0^\infty \trace\left(C \Phi(s) (B_1\mathcal W) (\hat B_1\mathcal W)^T\hat\Phi^T(s) \hat C^T\right) ds
                                                       \end{align*}
using the properties of the Frobenius norm. Due to the linearity of the trace and the integral, we find    
\begin{align*}
 \left\| H - \hat H\right\|_{\mathcal L^2(\mathcal W)}^2 = \trace(C P C^T) +   \trace(\hat C \hat P {\hat C}^T) - 2 \trace(C P_2 {\hat C}^T),                                               
                                                       \end{align*}
where we set \begin{align*}
 P:= &\int_0^\infty \mathbb E\left[\Phi(s) (B_1\mathcal W) (B_1\mathcal W)^T\Phi^T(s)\right] ds, \quad\hat P:= \int_0^\infty \mathbb E\left[\hat \Phi(s) (\hat B_1\mathcal W) (\hat B_1\mathcal W)^T\hat \Phi^T(s)\right] ds,\\
 &\quad\quad \quad\quad\quad \quad\quad P_2:= \int_0^\infty \mathbb E\left[\Phi(s) (B_1\mathcal W) (\hat B_1\mathcal W)^T\hat \Phi^T(s)\right] ds.
             \end{align*} 
Notice that $P, \hat P$ and $P_2$ exist due to Assumption \ref{assumption} for both systems. Now, the function $t\mapsto X(t):=\mathbb E\left[\Phi(t) (B_1\mathcal W) (\hat B_1\mathcal W)^T\hat \Phi^T(t)\right]$      
solves \eqref{dglmixed} in Lemma \ref{lemdgl} with $s=0$ and $\hat L = \hat B_1\mathcal W$,  $L = B_1\mathcal W$. Integrating \eqref{dglmixed} over $[0, v]$ yields
\begin{align*}
X(v)- (B_1\mathcal W) (\hat B_1\mathcal W)^T =\int_0^v X(t) dt {\hat A}^T+A \int_0^v X(t) dt +\sum_{i, j = 1}^{m_2} N_i \int_0^v X(t) dt  {\hat N}_j^T\;k_{ij}.
\end{align*}
Taking the limit of $v\rightarrow \infty$ above and taking into account that $\lim_{v\rightarrow \infty} X(v) = 0$ by assumption, we obtain that $P_2$ solves \eqref{mixed_gram}. With the same arguments setting 
$\Phi = \hat \Phi$ and $B_1 = \hat B_1$ and vice versa, we see that $P$ and $\hat P$ satisfy \eqref{full_gram} and \eqref{red_gram}, respectively.
\end{proof}
If $m_1 = m_2$ and if we replace the noise in \eqref{stochstate_mul} and \eqref{red_stochstate_mul} by the components $u_i$ of the control vector $u$, i.e., $dM_i(t)$ is replaced by $u_i(t) dt$, then these systems are so-called
(deterministic) bilinear systems. Let us denote these resulting bilinear systems by $\Sigma_{bil}$ and $\hat \Sigma_{bil}$ for the full and reduced system, respectively. Then, an $\mathcal H_2$-norm can be introduced for the bilinear case, too. We refer to 
\cite{morZhaL02} for more details. In \cite{morZhaL02} a Gramian based representation for $\left\|\Sigma_{bil} - \hat \Sigma_{bil} \right\|_{\mathcal H_2}$ is given. Interestingly, this distance coincides with 
the one in Theorem \ref{thm_stoch_h2_rep} if the noise and the input dimension coincide ($m_1 = m_2$), and the covariance and weight matrices are the identity ($K=(k_{ij}) = I$ and $\mathcal W = I$).
Consequently, a special case of $\left\| H - \hat H\right\|_{\mathcal L^2(\mathcal W)}$ yields a stochastic time-domain representation for the metric induced by the $\mathcal H_2$-norm for bilinear systems,
e.g., if $M$ is an $m_1$-dimensional standard Wiener process. We refer to \cite{redmann_h2_bil} for further connections between bilinear and stochastic linear systems. The next proposition deals with 
the distance between the outputs $y$ and $\hat y$, defined in \eqref{stochout_mul} and \eqref{red_stochout_mul}, and the above stochastic $\mathcal H_2$-distance $\left\| H - \hat H\right\|_{\mathcal L^2(\mathcal W)}$ when 
$\mathcal W = I$.
\begin{prop}\label{rel_output_error}
Let $y$ and $\hat y$ be the outputs of systems \eqref{original_system_mul} and \eqref{red_system_mul} with $x_0=0$, $\hat x_0=0$ 
and let $H= C \Phi(\cdot) B_1$, $\hat H= \hat C \hat \Phi(\cdot) \hat B_1$ the impulse responses of these 
systems. Then, for $T>0$, we have
\begin{align}
  \sup_{t \in [0, T]} \mathbb E \left\|y(t) - \hat y(t) \right\| \leq   \left\| H - \hat H\right\|_{\mathcal L^2(I)} \left\| u\right\|_{L^2_T}. \label{eq:boundstoch}      
\end{align}
\end{prop}
\begin{proof}
We find a solution representation for \eqref{stochstate_mul} by \begin{align*}
x(t) = \Phi(t) x_0 + \int_0^t \Phi(t, s) B_1 u(s) ds,
\end{align*}
where $\Phi(t, s) = \Phi(t)\Phi^{-1}(s)$ satisfies \eqref{funddef}. This is obtained by applying Ito's formula in \eqref{profriot} to $\Phi(t) g(t)$ with $g(t) := x_0 + \int_0^t \Phi^{-1}(s) B_1 u(s) ds$ using that 
$\Phi$ and $g$ are semimartingales, see Appendix \ref{appendixito}. Since $g$ is continuous and has a martingale part zero, \eqref{profriot} simply becomes the standard product rule due to \eqref{decomqucov} and one can 
show that $\Phi(t) g(t)$ is the solution to \eqref{stochstate_mul}. We proceed with the arguments used in \cite{redmannbenner}. The representations for the reduced and the full state with zero initial states yield 
\begin{align*}
 \mathbb E \left\|y(t) - \hat y(t) \right\| &\leq \mathbb E \int_0^t \left\|\left(C\Phi(t, s) B_1 - \hat C\hat \Phi(t, s) \hat B_1\right)u(s) \right\| ds\\  
 &\leq \mathbb E \int_0^t \left\|C\Phi(t, s) B_1 - \hat C\hat \Phi(t, s) \hat B_1\right\|_F \left\|u(s) \right\| ds.
\end{align*}
We apply the Cauchy-Schwarz inequality and obtain 
\begin{align*}
\mathbb E \left\|y(t) - \hat y(t) \right\| \leq \left(\mathbb E \int_0^t \left\|C\Phi(t, s) B_1 - \hat C\hat \Phi(t, s) \hat B_1\right\|_F^2 ds\right)^{\frac{1}{2}} \left\| u\right\|_{L^2_t}.
\end{align*}
Now, $\mathbb E[\Phi(t, s) B_1 \hat B_1 \hat \Phi(t, s)]  = \mathbb E[\Phi(t - s) B_1 \hat B_1 \hat \Phi(t - s)]$ due to \eqref{equamixedew}. Since the same property holds when considering 
$B_1 = \hat B_1$ and $\Phi = \hat \Phi$, we find  \begin{align*}
 \mathbb E \left\|y(t) - \hat y(t) \right\| &\leq \left(\mathbb E \int_0^t \left\|C\Phi(t - s) B_1 - \hat C\hat \Phi(t - s) \hat B_1\right\|_F^2 ds\right)^{\frac{1}{2}} \left\| u\right\|_{L^2_t}\\
 &= \left(\mathbb E \int_0^t \left\|C\Phi(s) B_1 - \hat C\hat \Phi(s) \hat B_1\right\|_F^2 ds\right)^{\frac{1}{2}} \left\| u\right\|_{L^2_t}.
\end{align*}
Taking the supremum on both sides and the upper bound of the integral to infinity implies the result.
\end{proof}
If  $\mathbb E\int_0^\infty \left\| u(s)\right\|^2 ds < \infty$, then we can replace $[0, T]$ by $[0, \infty)$ in Proposition \ref{rel_output_error}. The above result shows that one can expect a good approximation 
of \eqref{original_system_mul} by \eqref{red_system_mul} if matrices $\hat A, \hat N_i, \hat B_1, \hat C$ are chosen such that $\left\| H - \hat H\right\|_{\mathcal L^2(I)}$ is minimal. As mentioned above, 
$\left\| H - \hat H\right\|_{\mathcal L^2(I)}$ coincides with the $\mathcal H_2$-distance of the corresponding bilinear systems in special cases. In the bilinear case, MOR techniques have been considered that minimize 
the $\mathcal H_2$-error of two systems, e.g., bilinear IRKA \cite{breiten_benner}. Below, we construct related (${\mathcal L^2(\mathcal W)}$-optimal) algorithms which are very natural for stochastic systems due
to Proposition \ref{rel_output_error}, where $\mathcal W = I$. Note that for bilinear systems, the above proposition can only be shown if the right hand side is additionally multiplied by $\exp\left\{0.5 \left\| u\right\|^2_{L^2_T}\right\}$. 
This is a very recent result proved in \cite{redmann_h2_bil}. Therefore, considering IRKA type methods for stochastic systems seems even more natural than in the bilinear case in terms of the expected output error.
\begin{remark}
The error bound in Proposition \ref{rel_output_error} was shown for the special case of uncorrelated noise processes $M_i$ in \cite{redmannbenner} and used to determine an error bound for balanced truncation. 
The bound in \cite{redmannbenner} was also the basis for the proof of an error bound for another balancing related method in \cite{redSPA}. 
\end{remark}

\subsection{Conditions for $\mathcal L^2(\mathcal W)$-optimality and model order reduction in the multiplicative case}

Motivated by Proposition \ref{rel_output_error}, we locally minimize $\left\| H - \hat H\right\|_{\mathcal L^2(\mathcal W)}$ with respect to the coefficients of the reduced systems \eqref{red_system_mul} using the representation in Theorem \ref{thm_stoch_h2_rep}. 
In particular, we find necessary conditions for optimality for the expression 
\begin{align*}
\mathcal E(\hat A, \hat N_i, \hat B_1, \hat C) := \trace(\hat C \hat P {\hat C}^T) - 2 \trace(C P_2 {\hat C}^T),
\end{align*}
with $\hat P$ and $P_2$ solving \eqref{red_gram} and \eqref{mixed_gram}, respectively.
\begin{thm}\label{thm_opt_cond}
Consider the system \eqref{original_system_mul} with impulse response $H$. Let the reduced system \eqref{red_system_mul} with impulse response $\hat H$ be optimal with respect to $\mathcal L^2(\mathcal W)$, i.e., 
the matrices $\hat A, \hat N_i, \hat B_1, \hat C$ locally minimize the error $\left\| H - \hat H\right\|_{\mathcal L^2(\mathcal W)}$. Then, it holds that 
\begin{equation}\label{optcond}
\begin{aligned}
           (a)\quad&\hat C \hat P = C P_{2},\quad (b)\quad \hat Q\hat  P = Q_{2} P_{2},\\
           (c)\quad&\hat Q\left(\sum_{j=1}^{m_2}\hat N_j k_{ij} \right)\hat  P = Q_{2} \left(\sum_{j=1}^{m_2} N_j k_{ij}\right) P_{2},\quad i=1,\ldots, m_2,\\
(d)\quad &\hat Q\hat B_1 \mathcal W \mathcal W^T  =  Q_{2} B_1 \mathcal W \mathcal W^T,    
\end{aligned}
\end{equation}
where $\hat P, P_2$ are the solutions to \eqref{red_gram}, \eqref{mixed_gram} and $\hat Q, Q_2$ satisfy 
 \begin{align}\label{red_gram_obs}
                  \hat A^T \hat Q + \hat Q \hat A + \sum_{i, j=1}^{m_2} \hat N_i^T \hat Q \hat N_j k_{ij} &= - \hat C^T \hat C,\\ \label{mixed_gram_obs}
                  \hat A^T Q_2 + Q_2 A + \sum_{i, j=1}^{m_2} \hat N_i^T Q_2 N_j k_{ij} &= - \hat C^T C.
                  \end{align}
\end{thm}
\begin{proof}
If system \eqref{red_system_mul} is optimal with respect to the $\mathcal L^2(\mathcal W)$-norm, then we have  
\begin{align}\label{ifoptthen}
 \partial_z \mathcal E = 0 \Longleftrightarrow \partial_z \trace(\hat C \hat P {\hat C}^T)=  2 \partial_z \trace(C P_2 {\hat C}^T),                                    
\end{align}
where $z\in  \{\hat a_{km}, \hat n^{(l)}_{km}, \hat b_{kq}, \hat c_{\ell k}\}$, where $\hat A = (\hat a_{km})$, $\hat N_l = (\hat n^{(l)}_{km})$, $\hat B_1 = (\hat b_{kq})$ and $\hat C = (\hat c_{\ell k})$. Below, 
$e_k$ denotes the $k$th unit vector of suitable dimension. For $z = \hat c_{\ell k}$, \eqref{ifoptthen} becomes 
\begin{align*}
\trace(e_\ell e_k^T \hat P {\hat C}^T + {\hat C} \hat P e_k e_{\ell}^T)=  2 \trace(C P_2 e_k e_{\ell}^T).                                                                                                    
\end{align*}
Using the properties of the trace and ${\hat P}^T=\hat P$, this is equivalent to 
\begin{equation*}
e_{\ell}^T{\hat C} \hat P e_k=e_{\ell}^T C P_2 e_k
\end{equation*}
for all $\ell= 1,\ldots p$ and $k= 1, \ldots, r$. This results in equality (a): $\hat C \hat P = C P_{2}$. 
For $z = \hat a_{km}, \hat n^{(l)}_{km}, \hat b_{kq}$, \eqref{ifoptthen} becomes
\begin{align*}
\trace\left((\partial_{z} \hat P) {\hat C}^T \hat C \right)=  2 \trace\left((\partial_{z} P_2) {\hat C}^T C\right),
\end{align*}
which is equivalent to 
\begin{align*}
                        &\trace\left((\partial_{z} \hat P) (\hat A^T \hat Q + \hat Q \hat A + \sum_{i, j=1}^{m_2} \hat N_i^T \hat Q \hat N_j k_{ij}) \right)\\
                        &  =  2 \trace\left((\partial_{z} P_2) (\hat A^T Q_2 + Q_2 A + \sum_{i, j=1}^{m_2} \hat N_i^T Q_2 N_j k_{ij})\right) 
\end{align*}
using the equations for $\hat Q$ and $Q_2$. Again, by properties of the trace, the above can be reformulated as  
\begin{equation}\begin{aligned}\label{condabn}
                        &\trace\left(\left(\hat A (\partial_{z} \hat P)+(\partial_{z} \hat P) \hat A^T + \sum_{i, j=1}^{m_2} \hat N_i (\partial_{z} \hat P) \hat N_j^T k_{ij}\right) \hat Q \right)\\
                        &  =  2 \trace\left(\left( A (\partial_{z} P_2) + (\partial_{z} P_2) \hat A^T  + \sum_{i, j=1}^{m_2} N_i (\partial_{z} P_2)  \hat N_j^T k_{ij}\right) Q_2 \right),
                                                                                                            \end{aligned}
\end{equation}
taking into account that the covariance matrix $K=(k_{ij})$ is symmetric. We now derive equations for $\partial_{z} \hat P$ and $\partial_{z} P_2$ for each case.                                                                                                      
Applying $\partial_{\hat a_{km}}$ to \eqref{red_gram} and \eqref{mixed_gram}, we obtain 
\begin{align*}
e_k e_m^T \hat P + \hat A (\partial_{\hat a_{km}} \hat P) + (\partial_{\hat a_{km}} \hat P) {\hat A}^T+ \hat P e_m e_k^T +\sum_{i, j=1}^{m_2} \hat N_{i} (\partial_{\hat a_{km}} \hat P) {\hat N}_{j}^T k_{ij} &= 0,\\ 
A (\partial_{\hat a_{km}} P_2) + (\partial_{\hat a_{km}} P_2) {\hat A}^T+ P_2 e_m e_k^T +\sum_{i, j=1}^{m_2} N_{i} (\partial_{\hat a_{km}} P_2) {\hat N}_{j}^T k_{ij} &= 0.
    \end{align*}
Inserting this into \eqref{condabn}, and using symmetry of $\hat Q$ and $\hat P$ leads to 
\begin{align*}
 \trace\left((e_k e_m^T \hat P + \hat P e_m e_k^T) \hat Q \right) =  2 \trace\left((P_2 e_m e_k^T) Q_2 \right)   \Leftrightarrow  e_k^T \hat Q \hat P e_m  = e_k^T Q_2 P_2 e_m                                       
\end{align*}
for all $k, m = 1,\ldots, r$ which yields $\hat Q \hat P  =  Q_2 P_2$, i.e. equality (b). We now define $\hat\Psi_i^T := \sum_{j=1}^{m_2} {\hat N}_{j}^T k_{ij}$ and observe that 
$\partial_{\hat n^{(l)}_{km}} \hat\Psi_i^T = e_m e_k^T k_{il}$. Consequently, we have 
\begin{align*}
 \partial_{\hat n^{(l)}_{km}} \sum_{i, j=1}^{m_2} \hat N_i \hat P \hat N_j^T k_{ij} &=   \partial_{\hat n^{(l)}_{km}} \sum_{i=1}^{m_2} \hat N_i \hat P \hat\Psi_i^T  \\
 &= \sum_{i=1}^{m_2}\left( (\partial_{\hat n^{(l)}_{km}}\hat N_i) \hat P \hat\Psi_i^T + \hat N_i (\partial_{\hat n^{(l)}_{km}}\hat P) \hat\Psi_i^T + \hat N_i \hat P (\partial_{\hat n^{(l)}_{km}}\hat\Psi_i^T)\right)   \\
 &= e_k e_m^T \hat P \hat\Psi_l^T + \sum_{i=1}^{m_2} \hat N_i (\partial_{\hat n^{(l)}_{km}}\hat P) \hat\Psi_i^T + \sum_{i=1}^{m_2} \hat N_i \hat P e_m e_k^T k_{il} \\
  &= e_k e_m^T \hat P \hat\Psi_l^T+ \hat\Psi_l \hat P e_m e_k^T + \sum_{i, j=1}^{m_2} \hat N_i (\partial_{\hat n^{(l)}_{km}}\hat P)\hat N_j^T k_{ij},  
\end{align*}
since $k_{il}=k_{li}$. Analogue to the above steps, we obtain  
\begin{align*}
 \partial_{\hat n^{(l)}_{km}} \sum_{i, j=1}^{m_2} N_i P_2 \hat N_j^T k_{ij}&=  \partial_{\hat n^{(l)}_{km}} \sum_{i=1}^{m_2} N_i P_2 \hat\Psi_i^T \\
  &= \sum_{i=1}^{m_2}\left(N_i (\partial_{\hat n^{(l)}_{km}} P_2) \hat\Psi_i^T + N_i P_2 (\partial_{\hat n^{(l)}_{km}}\hat\Psi_i^T)\right)   \\
 &=   \Psi_l P_2 e_m e_k^T +  \sum_{i, j=1}^{m_2} N_i (\partial_{\hat n^{(l)}_{km}} P_2)\hat N_j^T k_{ij},
\end{align*}
where  $\Psi_l := \sum_{j=1}^{m_2} {N}_{j} k_{lj}$. Now, using these two results when applying $\partial_{\hat n^{(l)}_{km}}$ to \eqref{red_gram} and \eqref{mixed_gram}, we find 
\begin{align*}
               \hat A (\partial_{\hat n^{(l)}_{km}}\hat P) + (\partial_{\hat n^{(l)}_{km}} \hat P) {\hat A}^T+e_k e_m^T \hat P \hat\Psi_l^T+ \hat\Psi_l \hat P e_m e_k^T + \sum_{i, j=1}^{m_2} \hat N_i (\partial_{\hat n^{(l)}_{km}}\hat P)\hat N_j^T k_{ij}&= 0,\\ 
A (\partial_{\hat n^{(l)}_{km}}P_2) + (\partial_{\hat n^{(l)}_{km}} P_2) {\hat A}^T+\Psi_l P_2 e_m e_k^T + \sum_{i, j=1}^{m_2} N_i (\partial_{\hat n^{(l)}_{km}} P_2)\hat N_j^T k_{ij} &= 0.                                                                                                                
\end{align*}
We plug these into \eqref{condabn} resulting in
\begin{align*}
\trace\left((e_k e_m^T \hat P \hat\Psi_l^T+ \hat\Psi_l \hat P e_m e_k^T ) \hat Q \right)  =  2 \trace\left( \Psi_l P_2 e_m e_k^T  Q_2 \right) 
\end{align*}
for all $k, m = 1,\ldots, r$ and $l= 1, \ldots, m_2$. With the above arguments, and symmetry of $\hat P$ and $\hat Q$, this is equivalent to $\hat Q \hat\Psi_l \hat P  =   Q_2 \Psi_l P_2$ (equality (c)). It remains to apply $\partial_{\hat b_{kq}}$ to \eqref{red_gram}
and \eqref{mixed_gram} providing \begin{align*}                                                                                                      
 \hat A(\partial_{\hat b_{kq}}\hat P) + (\partial_{\hat b_{kq}}\hat P) {\hat A}^T+\sum_{i, j=1}^{m_2} \hat N_{i} (\partial_{\hat b_{kq}}\hat P) {\hat N}_{j}^T k_{ij} &= -(e_k e_q^T \mathcal W \mathcal W^T \hat B_1^T
 +\hat B_1 \mathcal W \mathcal W^T e_q e_k^T),\\ 
A (\partial_{\hat b_{kq}}P_2) + (\partial_{\hat b_{kq}} P_2) {\hat A}^T+\sum_{i, j=1}^{m_2} N_{i} (\partial_{\hat b_{kq}} P_2) {\hat N}_{j}^T k_{ij} &= -B_1 \mathcal W \mathcal W^T e_q e_k^T.
    \end{align*}                                            
Using this for \eqref{condabn}, we have   
\begin{align*}
  \trace\left((e_k e_q^T \mathcal W \mathcal W^T \hat B_1^T+\hat B_1 \mathcal W \mathcal W^T e_q e_k^T) \hat Q \right)  =  2 \trace\left( B_1 \mathcal W \mathcal W^T e_q e_k^T  Q_2 \right)
\end{align*}   
for all $k= 1,\ldots, r$ and $q=1, \ldots, m_1$. This results in the last property (d), i.e.     
$\hat Q \hat B_1 \mathcal W \mathcal W^T  =   Q_2 B_1 \mathcal W \mathcal W^T$ and concludes the proof.
\end{proof}
\begin{remark}
With the observations made in Section \ref{stochH2}, it is not surprising that setting $m_1 = m_2$, $K=(k_{ij}) = I$ and $\mathcal W = I$ in \eqref{optcond} leads to the necessary first-order optimality conditions in the deterministic 
bilinear case \cite{morZhaL02}. If $N_i = 0$ for all $i=1, \ldots, m_2$, then \eqref{optcond} represents a special case of weighted $\mathcal H_2$-optimality conditions in the deterministic linear setting \cite{halevi}, where 
the weight is a constant matrix. If further $\mathcal W=I$, then one obtains the Wilson conditions \cite{Wilson}.
\end{remark}
With the link between the $\mathcal L^2(\mathcal W)$-norm and the bilinear $\mathcal H_2$-norm given by Theorems \ref{thm_stoch_h2_rep} and \ref{thm_opt_cond}, it is now intuitively clear how to construct a reduced order system 
\eqref{red_system_mul} that satisfies \eqref{optcond}. The approach is oriented on bilinear IRKA \cite{breiten_benner} which fulfills the necessary conditions for local optimality given in \cite{morZhaL02}. Its modified version 
designed to satisfy \eqref{optcond} is provided in Algorithm \ref{algo:MBIRKA}. This algorithm requires that the matrix $\hat A$ is diagonalizable which we assume throughout this paper. In many applications this can be guaranteed.
\begin{algorithm}[!tb]
	\caption{Modified Bilinear IRKA}
	\label{algo:MBIRKA}
	\begin{algorithmic}[1]
		\Statex {\bf Input:} The system matrices: $ A, B_1, C, N_i$. Covariance and weight matrices: $K, \mathcal W$.
		\Statex {\bf Output:} The reduced matrices: $\hat A, \hat B_1,\hat C, \hat N_i$.
		\State Make an initial guess for the reduced matrices $\hat A, \hat B_1,\hat C, \hat N_i$.
\While {not converged}
		\State Perform the spectral decomposition of $\hA$ and define:
		\Statex\quad\qquad $D = S\hat A S^{-1},~\tilde B_1 = S\hat B_1, ~\tilde C = \hat C S^{-1}, ~\tilde N_i = S\hat N_i S^{-1}.$
		\State Solve for $V$ and $W$:
		
		\Statex \quad\qquad$ V D  +  AV + \sum_{i, j=1}^{m_2} N_i V \tilde N_j^T k_{ij}  = -(B_1\mathcal W)(\tilde B_1\mathcal W)^T$,
		\Statex \quad\qquad$ W D  +  A^T W + \sum_{i, j=1}^{m_2} N_i^T W \tilde N_j k_{ij} = -C^T\tilde C$.
\State $V = \orth{(V)}$ and $W = \orth{(W)}$, where $\orth{(\cdot)}$ returns an orthonormal basis for the range of a matrix.
		\State Determine the reduced matrices:
		\Statex \quad\qquad$\hat A = (W^T V)^{-1}W^T AV,\quad \hat B_1 = (W^T V)^{-1}W^T B_1,\quad\hat C = CV$, \quad $\hat N_i = (W^T V)^{-1}W^T N_i V$.
		\EndWhile
	\end{algorithmic}
\end{algorithm}
The following theorem proves that Algorithm \ref{algo:MBIRKA} provides reduced matrices that satisfy \eqref{optcond}. We shall later apply Algorithm \ref{algo:MBIRKA} with $\mathcal W = I$ to \eqref{original_system_mul}
in order to obtain a small output error for the resulting reduced system \eqref{red_system_mul} using Proposition \ref{rel_output_error}.
\begin{thm}\label{thm:error_opt_cond}
Let $\hat A$, $\hat N_i$, $\hat B$ and $\hat C$ be the reduced-order matrices computed by Algorithm~\ref{algo:MBIRKA} assuming that it converged. Then, $\hat A$, $\hat N_i$, $\hat B$ and $\hat C$ satisfy 
the necessary conditions \eqref{optcond} for local $\mathcal L^2(\mathcal W)$-optimality.
\end{thm}
\begin{proof}
The techniques used to prove the result are similar to the ones used in \cite{breiten_benner}. We provide the proof in Appendix \ref{proof_opt_cond} but we restrict ourselves to the conditions related 
to $\hat N_i$ because these differ significantly from the respective bilinear conditions.
\end{proof}

Finally, we note that, so far, many balancing related techniques for stochastic systems \eqref{original_system_mul} with multiplicative noise have been studied \cite{beckerhartmann, bennerdamm, redmannbenner, dammbennernewansatz, redmannspa2, redSPA}. Algorithm \ref{algo:MBIRKA} is the first alternative to such techniques for stochastic systems.

%% file: add.tex
\section{Systems with additive noise}
\label{sec:add}

We now focus on stochastic systems with additive noise. These are of the form \begin{subequations}\label{original_system_add}
\begin{align}\label{stochstate_add}
 dx(t) &= [Ax(t) +  B_1 u(t)]dt + B_2 dM(t),\quad x(0) = x_0,\\ \label{stochout_add}
y(t) &= Cx(t),\quad t\geq 0,
\end{align}
\end{subequations}
for which \eqref{stab_cond_add} is assumed. We find a ROM based on the minimization of an error bound for this case. 
Fortunately, we do not have to repeat the entire theory again since it can be derived from the above results for systems
with multiplicative noise. However, we will provide two different approaches to reduce system \eqref{original_system_add}. The first one relies on splitting and reducing subsystems of \eqref{original_system_add} individually and 
subsequently obtain the reduced system as the sum of the reduced subsystems. In the second approach, \eqref{original_system_add} will be reduced directly. We assume that \eqref{stab_cond_add} holds for all reduced systems below.

\subsection{Two step reduced order model}

We can write the state in \eqref{stochstate_add} as $x = x_1 + x_2$, where $x_1$ and $x_2$ are solutions to subsystems with corresponding outputs $y_1 = Cx_1$ and 
$y_2 = Cx_2$. Hence, we can rewrite \eqref{original_system_add} 
as follows:
\begin{subequations}\label{original_system_add_subsystem}
\begin{align}\label{stochstate_add_sub1}
 dx_1(t) &= [Ax_1(t) +  B_1 u(t)]dt,\quad x_1(0) = x_0,\\ \label{stochstate_add_sub2}
  dx_2(t) &= Ax_2(t)dt + B_2 dM(t),\quad x_2(0) = 0, \\
 \label{stochout_add_subsystems}
y(t) &= y_1(t) + y_2(t) = Cx_1(t)+ C x_2(t),\quad t\geq 0.
\end{align}
\end{subequations}
Now, the idea is to reduce \eqref{stochstate_add_sub1} and  \eqref{stochstate_add_sub2} with their associated outputs separately resulting in the following reduced system 
\begin{subequations}\label{red_system_add_subsystem}
\begin{align}\label{red_stochstate_add_sub1}
 d\hat x_1(t) &= [\hat A_1 \hat x_1(t) +  \hat B_1 u(t)]dt,\quad \hat x_1(0) = \hat x_0,\\ \label{red_stochstate_add_sub2}
  d\hat x_2(t) &= \hat A_2 \hat x_2(t)dt + \hat B_2 dM(t),\quad \hat x_2(0) = 0, \\
 \label{red_stochout_add_subsystems}
\hat y(t) &= \hat y_1(t) + \hat y_2(t) = \hat C_1 \hat x_1(t)+ \hat C_2\hat x_2(t),\quad t\geq 0,
\end{align}
\end{subequations}
where $\hat x_i(t)\in \mathbb R^{r_i}$ ($i=1, 2$) etc. with $r_i\ll n$. Note that there are several MOR techniques like balancing related schemes \cite{spa, moore} or $\mathcal H_2$-optimal methods \cite{linIRKA} to reduce the deterministic subsystem \eqref{stochstate_add_sub1}. 
Moreover, balancing related methods are available for \eqref{stochstate_add_sub2}, see \cite{hartmann2011, redmannfreitag}. We derive an optimization based scheme for \eqref{stochstate_add_sub2} below and combine it 
with \cite{linIRKA} for \eqref{stochstate_add_sub1} leading to a new type of method to reduce \eqref{original_system_add}.
\smallskip

The reduced system \eqref{red_system_add_subsystem} provides a higher flexibility since we are not forced to choose $\hat A_1 = \hat A_2 = \hat A$ and $\hat C_1 = \hat C_2 = \hat C$ as it would be the case 
if we apply an algorithm to \eqref{original_system_add} directly. Moreover, we are free in the choice of the dimension 
in each reduced subsystem. This is very beneficial since one subsystem might be easier to reduce than the other (note that one system is entirely deterministic, whereas the other one is stochastic). This additional flexibility is expected to give a better reduced system 
than by a reduction of \eqref{original_system_add} in one step. However, it is more expensive to run a model reduction procedure twice.\smallskip

We now explain how to derive the reduced matrices above, assuming $x_0=0$ and $\hat x_0 = 0$. Using the inequality of Cauchy-Schwarz, we obtain \begin{align*}
      \mathbb E \left\|y(t) - \hat y(t) \right\|  &\leq \mathbb E \left\|y_1(t) - \hat y_1(t) \right\|  +\mathbb E \left\|y_2(t) - \hat y_2(t) \right\| \\
      &\leq \mathbb E \left\|y_1(t) - \hat y_1(t) \right\|  +\left(\mathbb E \left\|y_2(t) - \hat y_2(t) \right\|^2\right)^{\frac{1}{2}}.                                                             
                                                                     \end{align*}
We insert the solution representations for both $x_1$ and $x_2$ as well as for their reduced systems into the above relation and find 
\begin{align*}
      \mathbb E \left\|y(t) - \hat y(t) \right\| \leq\, &\mathbb E \int_0^t\left\|\left(C\expn^{A(t-s)} B_1 - \hat C_1 \expn^{\hat A_1(t-s)} \hat B_1\right)u(s)\right\| ds  \\
      &+\left(\mathbb E \left\|\int_0^t\left(C\expn^{A(t-s)} B_2 - \hat C_2 \expn^{\hat A_2(t-s)} \hat B_2\right) dM(s)  \right\|^2\right)^{\frac{1}{2}}\\
      \leq\, &\left(\int_0^t \left\|C\expn^{A(t-s)} B_1 - \hat C_1\expn^{\hat A_1(t-s)} \hat B_1\right\|_F^2  ds\right)^{\frac{1}{2}}  \left\| u\right\|_{L^2_t}\\
      &+\left(\int_0^t\left\|\left(C\expn^{A(t-s)} B_2 - \hat C_2 \expn^{\hat A_2(t-s)} \hat B_2\right)K^{\frac{1}{2}}\right\|_F^2 ds  \right)^{\frac{1}{2}}\\
      \leq\, &\left(\int_0^\infty \left\|C\expn^{As} B_1 - \hat C_1 \expn^{\hat A_1s} \hat B_1\right\|_F^2  ds\right)^{\frac{1}{2}}  \left\| u\right\|_{L^2_t}\\
      &+\left(\int_0^\infty\left\|\left(C\expn^{As} B_2 - \hat C_2 \expn^{\hat A_2s} \hat B_2\right)K^{\frac{1}{2}}\right\|_F^2 ds  \right)^{\frac{1}{2}},
                                                                     \end{align*}
where we have applied the Cauchy-Schwarz inequality to the first error term and the Ito isometry to the second one (see, e.g., \cite{zabczyk}) and substituted $t-s \mapsto s$. Now, applying the supremum on $[0, T]$ to the above 
inequality yields
\begin{align}\label{error_est_add}
  \sup_{t\in [0, T]} \mathbb E \left\|y(t) - \hat y(t) \right\| \leq  \underbrace{\left\|C\expn^{A \cdot} B_1 - \hat C_1\expn^{\hat A_1 \cdot} \hat B_1\right\|_{\mathcal L^2(I)}}_{=:\mathcal E_1}\left\| u\right\|_{L^2_T}
      +\underbrace{\left\|C\expn^{A \cdot} B_2 - \hat C_2\expn^{\hat A_2 \cdot} \hat B_2\right\|_{\mathcal L^2(K^{\frac{1}{2}})}}_{=:\mathcal E_2}.
                                                                     \end{align}
In order for the right hand side to be small, the reduced order matrices need to be chosen such that $\mathcal E_1$ and $\mathcal E_2$ are locally minimal. We observe that $\mathcal E_1$ and $\mathcal E_2$
are special cases of the $\mathcal L^2(\mathcal W)$-distance of the impulse responses in the multiplicative noise scenario, see Section \ref{stochH2}, since $\Phi(t) = \expn^{A t}$ if $N_i=0$ for all $i = 1, \ldots, m_2$. Here, we have $N_i=0$, $\mathcal W= I$ for $\mathcal E_1$ and $N_i=0$, $\mathcal W= K^{\frac{1}{2}}$ with $B_1$ replaced by $B_2$ for $\mathcal{E}_2$. Consequently, Algorithm \ref{algo:MBIRKA} with $N_i=0$ and the respective choice for the weight matrices satisfies the necessary 
conditions for local optimality for $\mathcal E_1$ and $\mathcal E_2$. Therefore, $(\hat A_i, \hat B_i, \hat C_i)$ ($i=1, 2$) can be computed from Algorithm \ref{algo:IRKA} (a special version of Algorithm \ref{algo:MBIRKA}) which is a modified version of
linear IRKA \cite{linIRKA}.  
\begin{algorithm}[!tb]
	\caption{Modified Two Step Linear IRKA ($i=1, 2$)}
	\label{algo:IRKA}
	\begin{algorithmic}[1]
		\Statex {\bf Input:} The system matrices: $ A, B_i, C$. Weight: $\mathcal W_i = \begin{cases}
  I,  & i=1,\\
  K^{\frac{1}{2}}, & i=2.
\end{cases}$
		\Statex {\bf Output:} The reduced matrices: $\hat A_i, \hat B_i,\hat C_i$.
		\State Make an initial guess for the reduced matrices $\hat A_i, \hat B_i,\hat C_i$.
\While {not converged}
		\State Perform the spectral decomposition of $\hat A_i$ and define:
		\Statex\quad\qquad $D_i = S\hat A_i S^{-1},~\tilde B_i = S\hat B_i, ~\tilde C_i = \hat C_i S^{-1}.$
		\State Solve for $V$ and $W$:
		
		\Statex \quad\qquad$ V D_i  +  AV = -(B_i\mathcal W_i)(\tilde B_i\mathcal W_i)^T$,
		\Statex \quad\qquad$ W D_i  +  A^T W  = -C^T\tilde C_i$.
\State $V = \orth{(V)}$ and $W = \orth{(W)}$.
		\State Determine the reduced matrices:
		\Statex \quad\qquad$\hat A_i = (W^T V)^{-1}W^T AV,\quad \hat B_i = (W^T V)^{-1}W^T B_i,\quad\hat C_i = CV$.
		\EndWhile
	\end{algorithmic}
\end{algorithm}

\subsection{One step reduced order model}

The second approach for additive noise reduces \eqref{original_system_add} directly without dividing it into subsystems. To do so, we set
\begin{align}\label{set_coef_eq}
 \hat A_1 = \hat A_2 = \hat A,\quad \hat C_1 = \hat C_2 = \hat C
\end{align}
in \eqref{red_system_add_subsystem}. This results in the following reduced system 
 \begin{subequations}\label{red_system_add}
\begin{align}\label{red_stochstate_add}
 d\hat x(t) &= [\hat A\hat x(t) +  \hat B_1 u(t)]dt + \hat B_2 dM(t),\quad \hat x(0) = \hat x_0,\\ \label{red_stochout_add}
\hat y(t) &= \hat C\hat x(t),\quad t\geq 0,
\end{align}
\end{subequations}
where $\hat x(t)\in \mathbb R^{r}$ etc. with $r\ll n$. Again, we assume that $x_0=0$ and $\hat x_0=0$. We insert \eqref{set_coef_eq} into \eqref{error_est_add} and obtain
\begin{align}\nonumber
  &\sup_{t\in [0, T]} \mathbb E \left\|y(t) - \hat y(t) \right\| \leq  \left\|C\expn^{A \cdot} B_1 - \hat C\expn^{\hat A \cdot} \hat B_1\right\|_{\mathcal L^2(I)}\left\| u\right\|_{L^2_T}
      +\left\|C\expn^{A \cdot} B_2 - \hat C \expn^{\hat A \cdot} \hat B_2\right\|_{\mathcal L^2(K^{\frac{1}{2}})}\\ \label{one_tep_bound1}
&\leq\left(\left\|C\expn^{A \cdot} B_1 - \hat C\expn^{\hat A \cdot} \hat B_1\right\|_{\mathcal L^2(I)} + \left\|C\expn^{A \cdot} B_2 - \hat C \expn^{\hat A \cdot} \hat B_2\right\|_{\mathcal L^2(K^{\frac{1}{2}})}\right)
 \max\{1, \left\| u\right\|_{L^2_T}\}\\ \nonumber
 &\leq \sqrt{2} \left(\left\|C\expn^{A \cdot} B_1 - \hat C \expn^{\hat A \cdot} \hat B_1\right\|_{\mathcal L^2(I)}^2
 + \left\|C\expn^{A \cdot} B_2 - \hat C\expn^{\hat A \cdot} \hat B_2\right\|_{\mathcal L^2(K^{\frac{1}{2}})}^2\right)^{\frac{1}{2}}  \max\{1, \left\| u\right\|_{L^2_T}\}.
                                                                     \end{align}
Now, we exploit that $\left\|L_1\right\|_F^2 + \left\|L_2\right\|_F^2 = \left\|\mat{cc} L_1 & L_2\rix\right\|_F^2$ for matrices $L_1, L_2$ of suitable dimension. Hence, we have\begin{align*}
 &\left\|C\expn^{A s} B_1 - \hat C\expn^{\hat A s} \hat B_1\right\|_F^2  + \left\|C\expn^{As} B_2 K^{\frac{1}{2}} - \hat C \expn^{\hat As} \hat B_2 K^{\frac{1}{2}}\right\|_F^2\\
 &= \left\|C\expn^{A s} \mat{cc} B_1& B_2 K^{\frac{1}{2}}\rix - \hat C\expn^{\hat A s} \mat{cc} \hat B_1 &\hat B_2 K^{\frac{1}{2}}\rix\right\|_F^2 \\
 &= \left\|\left(C\expn^{A s} \mat{cc} B_1& B_2\rix - \hat C\expn^{\hat A s} \mat{cc} \hat B_1 &\hat B_2 \rix\right)\smat I & 0\\ 0 & K^{\frac{1}{2}}\srix\right\|_F^2.
                     \end{align*}
Plugging this into \eqref{one_tep_bound1} yields \begin{align}\label{one_step_bound}
  \sup_{t\in [0, T]} \mathbb E \left\|y(t) - \hat y(t) \right\| \leq    
  \sqrt{2} \underbrace{\left\|C\expn^{A \cdot} \mat{cc} B_1& B_2\rix - \hat C \expn^{\hat A \cdot} \mat{cc} \hat B_1 &\hat B_2 \rix\right\|_{\mathcal L^2(\mathcal W)}}_{=:\mathcal E_3} \max\{1, \left\| u\right\|_{L^2_T}\},
                                                 \end{align}
where $\mathcal W= \smat I & 0\\ 0 & K^{\frac{1}{2}}\srix$. We now want to find a ROM such that $\mathcal E_3$ is small leading to a small output error. Again, $\mathcal E_3$ is a special case 
of the impulse response error of a stochastic system with multiplicative noise, where $N_i=0$, $B_1$ is replaced by $B = \mat{cc}  B_1 & B_2 \rix$ and $\mathcal W= \smat I & 0\\ 0 & K^{\frac{1}{2}}\srix$, cf. Section \ref{stochH2}. Taking this into account in Algorithm \ref{algo:MBIRKA} we obtain a method that satisfies the necessary optimality conditions for $\mathcal E_3$. This method is given 
in Algorithm \ref{algo:IRKA2} and again represents a modified version of linear IRKA. We can therefore apply Algorithm \ref{algo:IRKA2} in order to compute the reduced matrices $(\hat A, \hat B_1, \hat B_2, \hat C)$ in \eqref{red_system_add}. 
This scheme is computationally cheaper than Algorithm \ref{algo:IRKA} but cannot be expected to perform in the same way. The reduced system  \eqref{red_system_add} is less flexible than \eqref{red_system_add_subsystem} in terms of 
the choice of the reduced order dimensions and coefficients and it furthermore relies on the minimization of a more conservative bound in \eqref{one_step_bound} in comparison to \eqref{error_est_add}. 
Note that with Algorithm \ref{algo:IRKA2}, an alternative method to applying balanced truncation to \eqref{original_system_add} (see, for example \cite{stochInhom}) has been found.
\begin{algorithm}[!tb]
	\caption{Modified One Step Linear IRKA}
	\label{algo:IRKA2}
	\begin{algorithmic}[1]
		\Statex {\bf Input:} The system matrices: $ A, B=\mat{cc} B_1 & B_2\rix, C$. Weight: $\mathcal W = \smat I & 0\\ 0 & K^{\frac{1}{2}}\srix$.
		\Statex {\bf Output:} The reduced matrices: $\hat A, \hat B=\mat{cc} \hat B_1 & \hat B_2\rix,\hat C$.
		\State Make an initial guess for the reduced matrices $\hat A, \hat B=\mat{cc} \hat B_1 & \hat B_2\rix,\hat C$.
\While {not converged}
		\State Perform the spectral decomposition of $\hat A_i$ and define:
		\Statex\quad\qquad $D = S\hat A S^{-1},~\tilde B = S\hat B, ~\tilde C = \hat C S^{-1}.$
		\State Solve for $V$ and $W$:
		
		\Statex \quad\qquad$ V D  +  AV = -(B\mathcal W)(\tilde B\mathcal W)^T = -B \smat I & 0\\ 0 & K\srix\tilde B^T$,
		\Statex \quad\qquad$ W D  +  A^T W  = -C^T\tilde C$.
\State $V = \orth{(V)}$ and $W = \orth{(W)}$.
		\State Determine the reduced matrices:
		\Statex \quad\qquad$\hat A = (W^T V)^{-1}W^T AV,\quad \mat{cc} \hat B_1 & \hat B_2\rix=\hat B = (W^T V)^{-1}W^T B,\quad\hat C = CV$.
		\EndWhile
	\end{algorithmic}
\end{algorithm}

%% file: numerics.tex
\section{Numerical experiments}
\label{sec:numerics}

We now apply Algorithm  \ref{algo:MBIRKA} to a large-scale stochastic differential equation with multiplicative noise as well as Algorithms \ref{algo:IRKA} and \ref{algo:IRKA2} to a high dimensional equation 
with additive noise. In both scenarios, the examples are spatially discretized versions of controlled stochastic damped wave equations. These equations are extensions of the numerical examples considered 
in \cite{redbensec,redmannfreitag}. In particular, we study 
\begin{align*}
 \frac{\partial^2}{\partial t^2} X(t, z)+\alpha \frac{\partial}{\partial t} 
 X(t, z)=\frac{\partial^2}{\partial z^2} X(t,z)+f_1(z)u(t)+
 \begin{cases}
  \sum_{i=1}^2 f_{2, i}(z) \frac{\partial}{\partial t} w_i(t),\\
  \sum_{i=1}^2 g_{i}(z) X(t, z) \frac{\partial}{\partial t} w_i(t), 
\end{cases}
 \end{align*}
for $t\in[0, T]$ and $z\in[0, \pi]$, and $w_1$ and $w_2$ are standard Wiener processes that are correlated. Boundary and initial conditions are given by \begin{align*}
  X(0, t)=0=X(\pi, t)\quad\text{and}\quad X(0, z), \left.\frac{\partial}{\partial t} X(t, z)\right\vert_{t=0}\equiv 0.
\end{align*}
We assume that the quantity of interest is either the position of the midpoint 
\begin{align*}
Y(t)=\frac{1}{2\epsilon}\int_{\frac{\pi}{2}-\epsilon}^{\frac{\pi}{2}+\epsilon} X(t,z) dz,
\end{align*}
or the velocity of the midpoint 
\begin{align*}
Y(t)=\frac{1}{2\epsilon}\int_{\frac{\pi}{2}-\epsilon}^{\frac{\pi}{2}+\epsilon} \frac{\partial}{\partial t}X(t,z) dz,
\end{align*}
where $\epsilon>0$ is small. We can transform the above wave equation into a first order system and discretize it using a spectral Galerkin method as in \cite{redbensec,redmannfreitag}. This leads to 
the following stochastic differential equations
\begin{subequations}\label{example_SDE}
\begin{align}
 dx(t) &= [Ax(t) +  B_1 u(t)]dt + \begin{cases}
  B_2 \smat dw_1(t) & dw_2(t)\srix^T,\\
  \sum_{i=1}^{2} N_i x(t) dw_i(t),
\end{cases}\\ 
y(t) &= Cx(t),\quad x(0) = 0,\quad t\in [0, T],
\end{align}
\end{subequations}
where $y \approx Y$ if the dimension $n$ of $x$ is sufficiently large. Let $n$ be even. Then, for $\ell= 1, \ldots, \frac{n}{2}$, the associated matrices are
\begin{itemize}
\item{$A=\diag \left(E_1, \ldots, 
E_{\frac{n}{2}}\right)$ with $E_\ell=\left( \begin{smallmatrix} 0 & \ell\\ 
-\ell &-\alpha \end{smallmatrix}\right)$,}
\item{$B_1=\left(b^{(1)}_k\right)_{k=1, \ldots, n}$ with 
\begin{align*}
b^{(1)}_{2\ell-1}=0,\quad b^{(1)}_{2\ell}=\sqrt{\frac{2}{\pi}} \left\langle 
f_1, \sin(\ell\cdot) \right\rangle_{L^2{([0, \pi])}},\end{align*}} 
\item{$B_2=\mat{cc} B_2^{(1)} &B_2^{(2)} \rix$, where  $B_2^{(i)}=\left(b^{(2, i)}_k\right)_{k=1, \ldots, n}$ with 
\begin{align*}
b^{(2, i)}_{2\ell-1}=0,\quad b^{(2, i)}_{2\ell}=\sqrt{\frac{2}{\pi}} \left\langle 
f_{2, i}, \sin(\ell\cdot) \right\rangle_{L^2{([0, \pi])}},\end{align*}}
\item {$N_i=\left( n^{(i)}_{kj}\right)_{k, j=1, \ldots, n}$ with \begin{align*}
n^{(i)}_{(2 \ell-1)j}=0,\quad n^{(i)}_{(2 \ell)j}=\begin{cases}
  0,  & \text{if } j=2v,\\
  \frac{2}{\pi v}\left\langle 
\sin(\ell\cdot), g_i\sin(v\cdot) \right\rangle_{L^2{([0, \pi])}}, & \text{if }j=2v-1,
\end{cases}\end{align*} for $i=1, 2$, $j=1, \ldots, n$ and $v=1, \ldots, \frac{n}{2}$,
} 
\item{$C^T=\left(c_k\right)_{k=1, \ldots, n}$ with \begin{align*}
c_{2 \ell}=0\quad c_{2\ell-1}=\frac{1}{\sqrt{2\pi}\ell^2\epsilon}\left[ \cos \left(\ell \left(\frac{\pi}{2}
-\epsilon \right)\right)-\cos\left(\ell\left(\frac { \pi } { 2 }+\epsilon 
\right)\right)\right],\end{align*}
if we are interested in the position as output and 
\begin{align*}
c_{2 \ell-1}=0\quad c_{2\ell}=\frac{1}{\sqrt{2\pi}\ell\epsilon}\left[ \cos \left(\ell \left(\frac{\pi}{2}
-\epsilon \right)\right)-\cos\left(\ell\left(\frac { \pi } { 2 }+\epsilon 
\right)\right)\right],\end{align*}
if we are interested in the velocity as output.
}
\end{itemize}
For the following examples we choose $n = 1000$, $T = 1$ and the correlation $\mathbb E [w_1(t) w_2(t)] = 0.5 t$ meaning that $K=\smat 1 & 0.5 \\ 0.5 & 1 \srix$.

\paragraph{Multiplicative noise}
We start with the multiplicative case in \eqref{example_SDE} and compute the ROM \eqref{red_system_mul} by Algorithm \ref{algo:MBIRKA}. We use $\alpha=2$, $f_1(z)=\sin(3z)$, $g_{1}(z) = \expn^{-(z-\frac{\pi}{2})^2}$ and $g_{2}(z) = \expn^{-\frac{1}{2}(z-\frac{\pi}{2})^2}$. As input function we take $u(t) = e^{-0.1t}$. We compute a ROM of dimension $r=6$ using the modified bilinear IRKA algorithm. 
\begin{figure}[ht]
\begin{minipage}[b]{0.45\linewidth}
\centering
\includegraphics[width=\textwidth,height = 5cm]{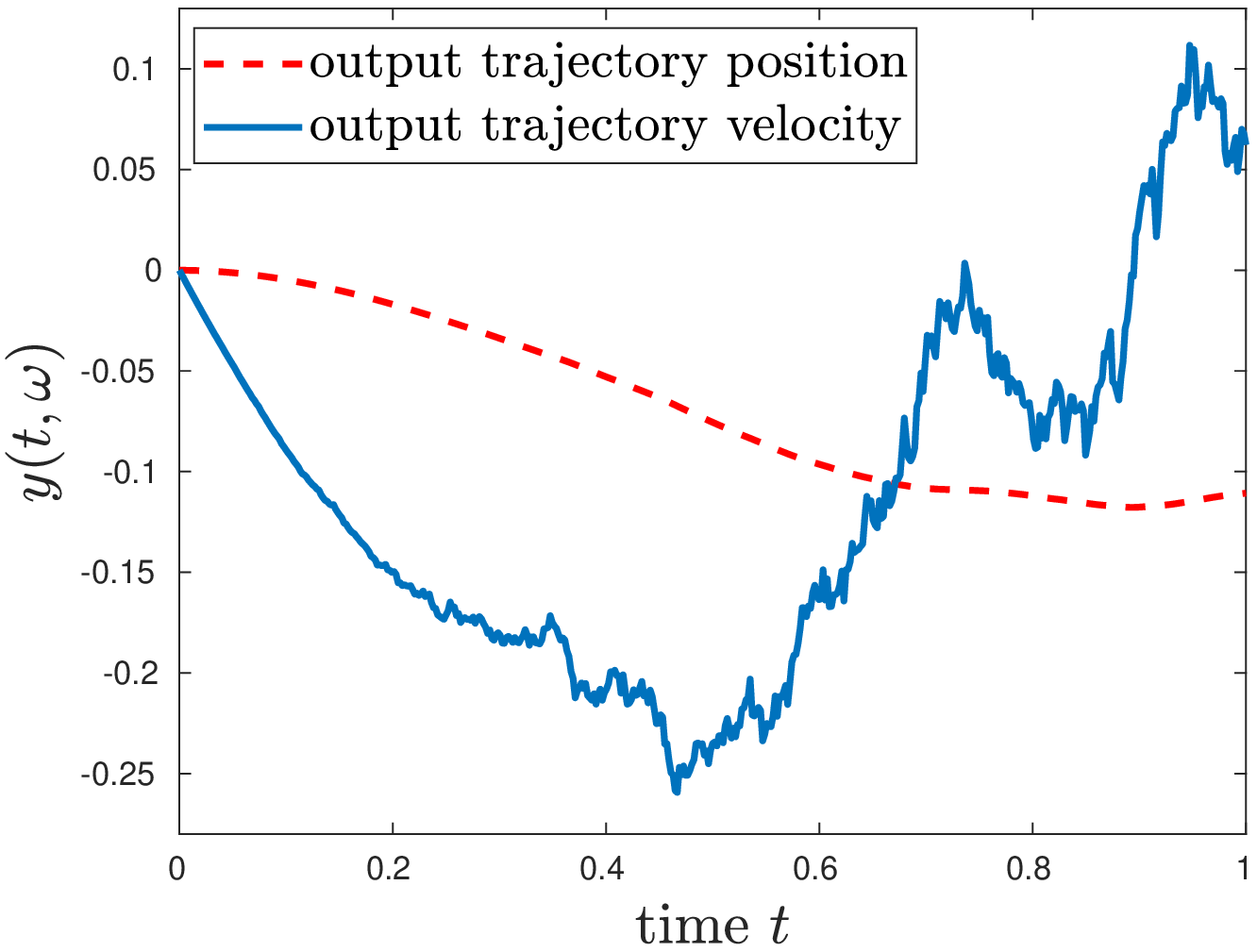}
\caption{Trajectory of position and velocity in the middle of the string.}
\label{fig:samplemult}
\end{minipage}
\hspace{0.5cm}
\begin{minipage}[b]{0.45\linewidth}
\centering
\includegraphics[width=\textwidth, height = 5cm]{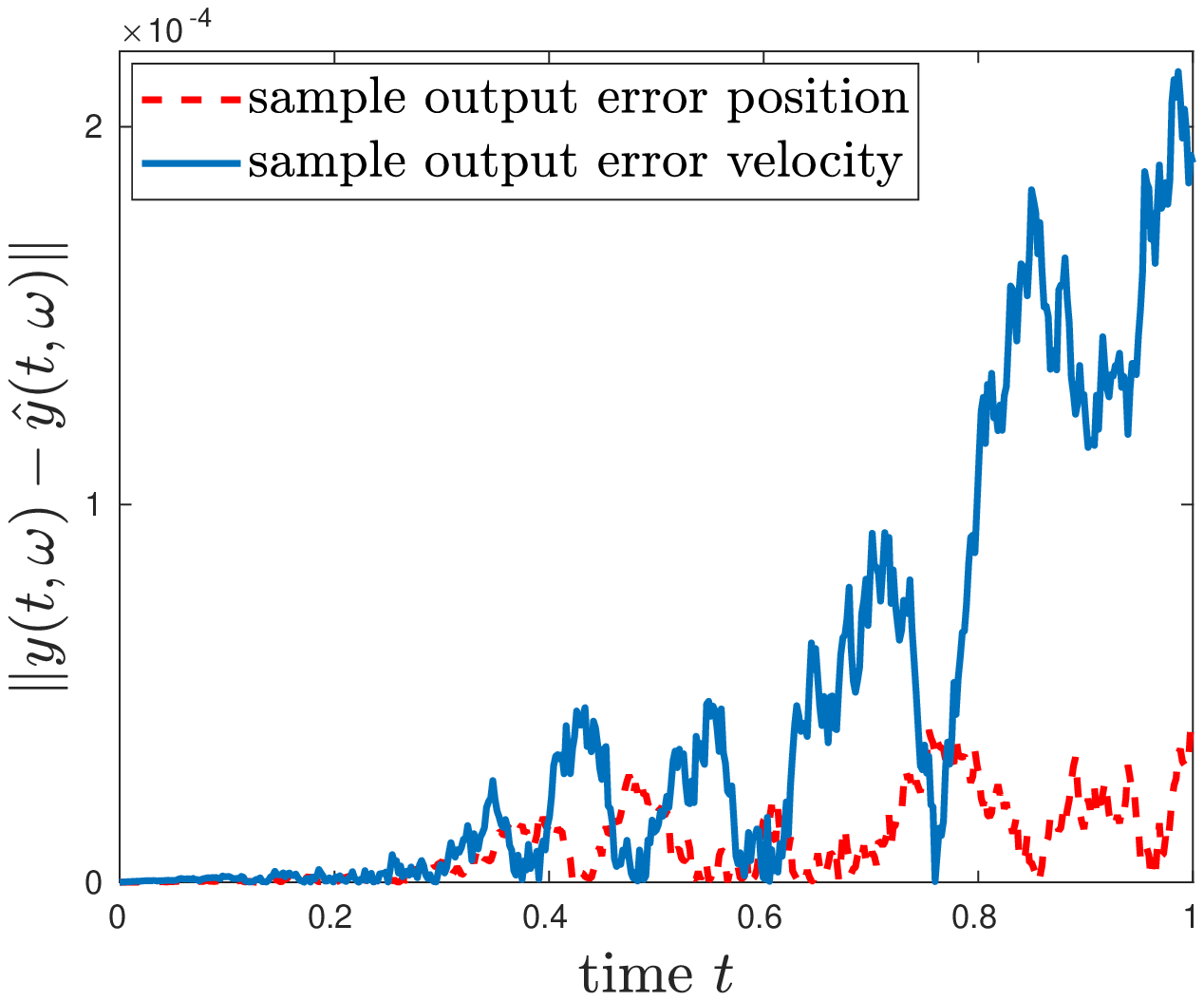}
\caption{Trajectory errors between full model and ROM of size $r=6$.}
\label{fig:samplemulterror}
\end{minipage}
\end{figure}
\begin{figure}[ht]
\begin{minipage}[b]{0.45\linewidth}
\centering
\includegraphics[width=\textwidth, height = 5cm]{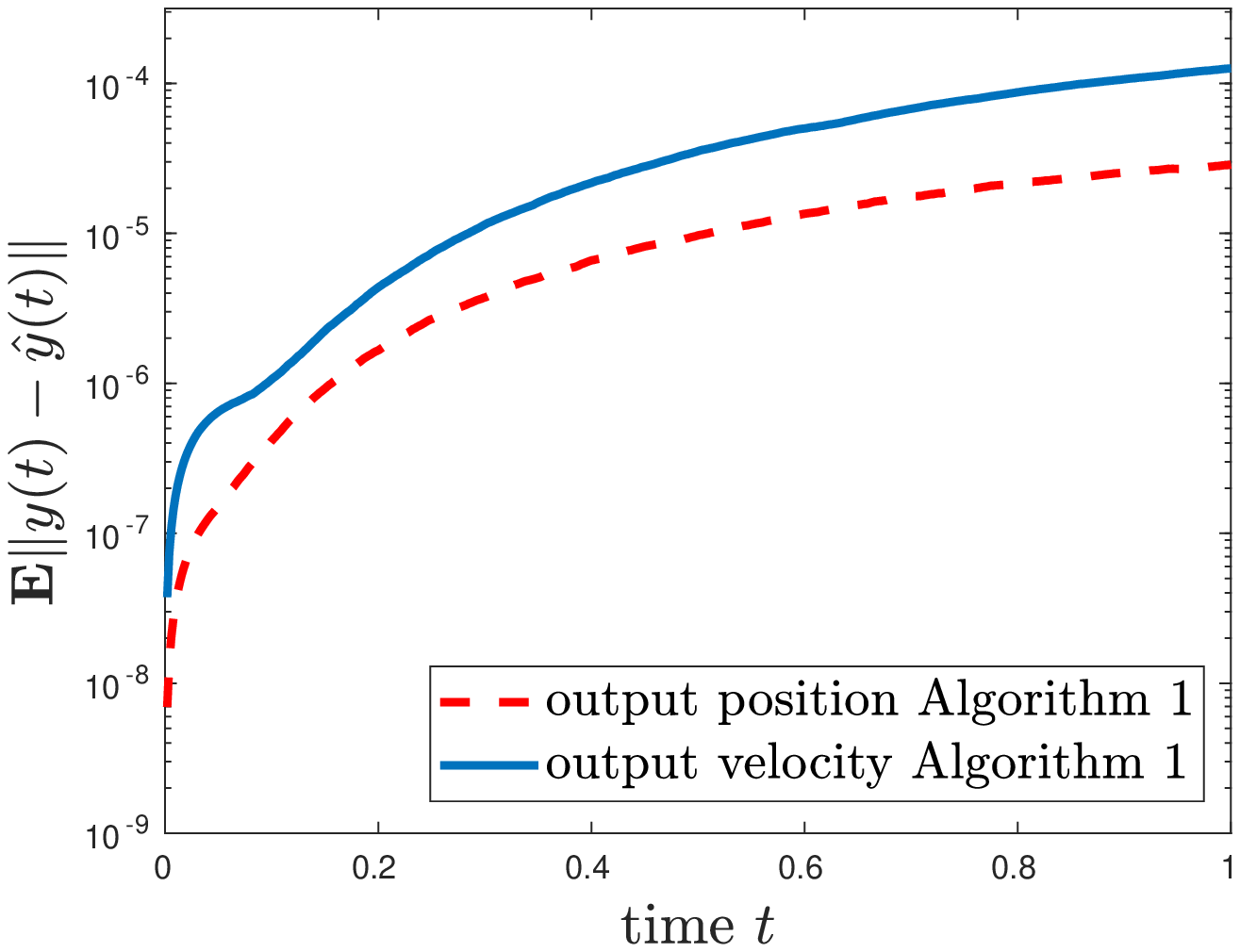}
\caption{Mean errors for position and velocity for $r=6$.}
\label{fig:C1C2expect}
\end{minipage}
\hspace{0.5cm}
\begin{minipage}[b]{0.45\linewidth}
\centering
\includegraphics[width=\textwidth,height = 5cm]{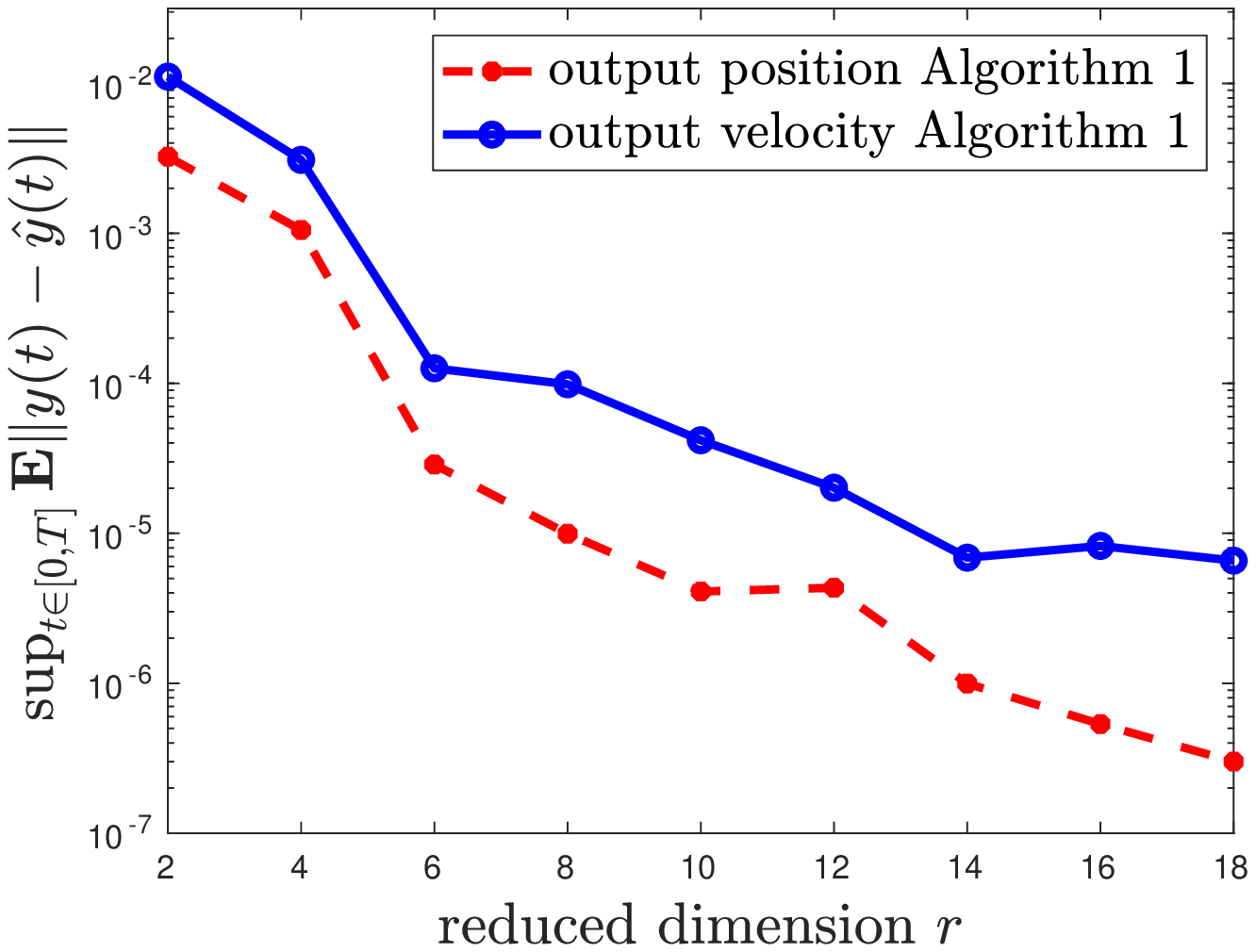}
\caption{Worst case mean error for several dimensions of the ROM.}
\label{fig:samplemulterror_log}
\end{minipage}
\end{figure}
Figure \ref{fig:samplemult} shows the output trajectory, i.e., the position and the velocity in the middle of the string for one particular sample over the time window $[0,1]$. The trajectory of the 
velocity is as rough as the noise process, whereas the trajectory for the position is smooth as it is the integral of the velocity. Figure \ref{fig:samplemulterror} plots the respective point-wise errors 
 between the full model of dimension $n=1000$ and ROM of dimension $r=6$ for fixed trajectories. In Figure \ref{fig:C1C2expect} the expected value of the output error of the full and the reduced model
 is plotted. Both in Figure \ref{fig:samplemulterror} for the sample error and in Figure \ref{fig:C1C2expect} for the mean error we observe that the output error is smaller for the position than for the velocity as this
 is a smoother function. 

Finally, we use Algorithm \ref{algo:MBIRKA} in order to compute several ROMs of dimensions $r=2,\ldots,18$ and the corresponding worst case error 
$\sup_{t \in [0, T]} \mathbb E \left\|y(t) - \hat y(t) \right\|$, which we plotted in Figure \ref{fig:samplemulterror_log}. We observe that the error decreases as the size of the ROM increases, as one would expect. We also see that the output error in the position is consistently about one magnitude smaller than the output error in the velocity.

\paragraph{Additive noise} For the additive case in \eqref{example_SDE} we use $\alpha=0.1$, $f_1(z)=\cos(2z)$, $f_{2,1}(z) = \sin(z)$ and 
$f_{2,2}(z) = \sin(z)\exp(-(z-\pi/2)^2)$. As input function we take $u \equiv 1$, such that $\|u\|_{L_T^2}=1$. For systems with additive noise \eqref{stochstate_add} we compare the two approaches considered in this
paper for computing a ROM. In this example we only consider the position for the output. Qualitatively we obtain the same results for the velocity, the error is typically larger by one magnitude, 
as we have seen for the case of multiplicative noise. 
\begin{figure}[ht]
\begin{minipage}[b]{0.45\linewidth} 
\centering
\includegraphics[width=\textwidth, height = 5cm]{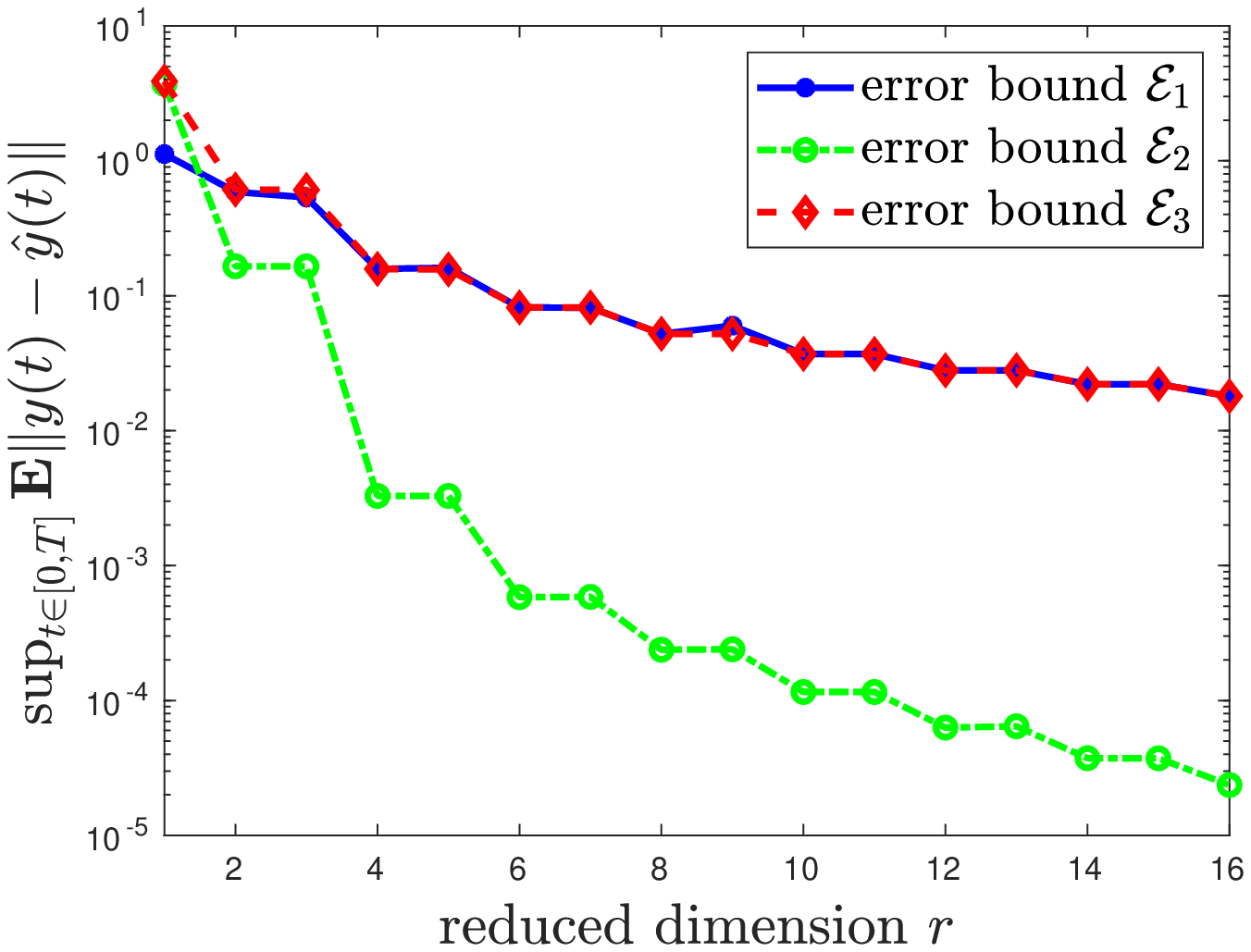}
\caption{Bounds $\mathcal{E}_1$, $\mathcal{E}_2$ and $\mathcal{E}_3$ from \eqref{error_est_add} and \eqref{one_step_bound} for different dimensions of the ROM.}
\label{fig:errorbound_add}
\end{minipage}
\hspace{0.5cm}
\begin{minipage}[b]{0.45\linewidth}
\centering
\includegraphics[width=\textwidth,height = 5cm]{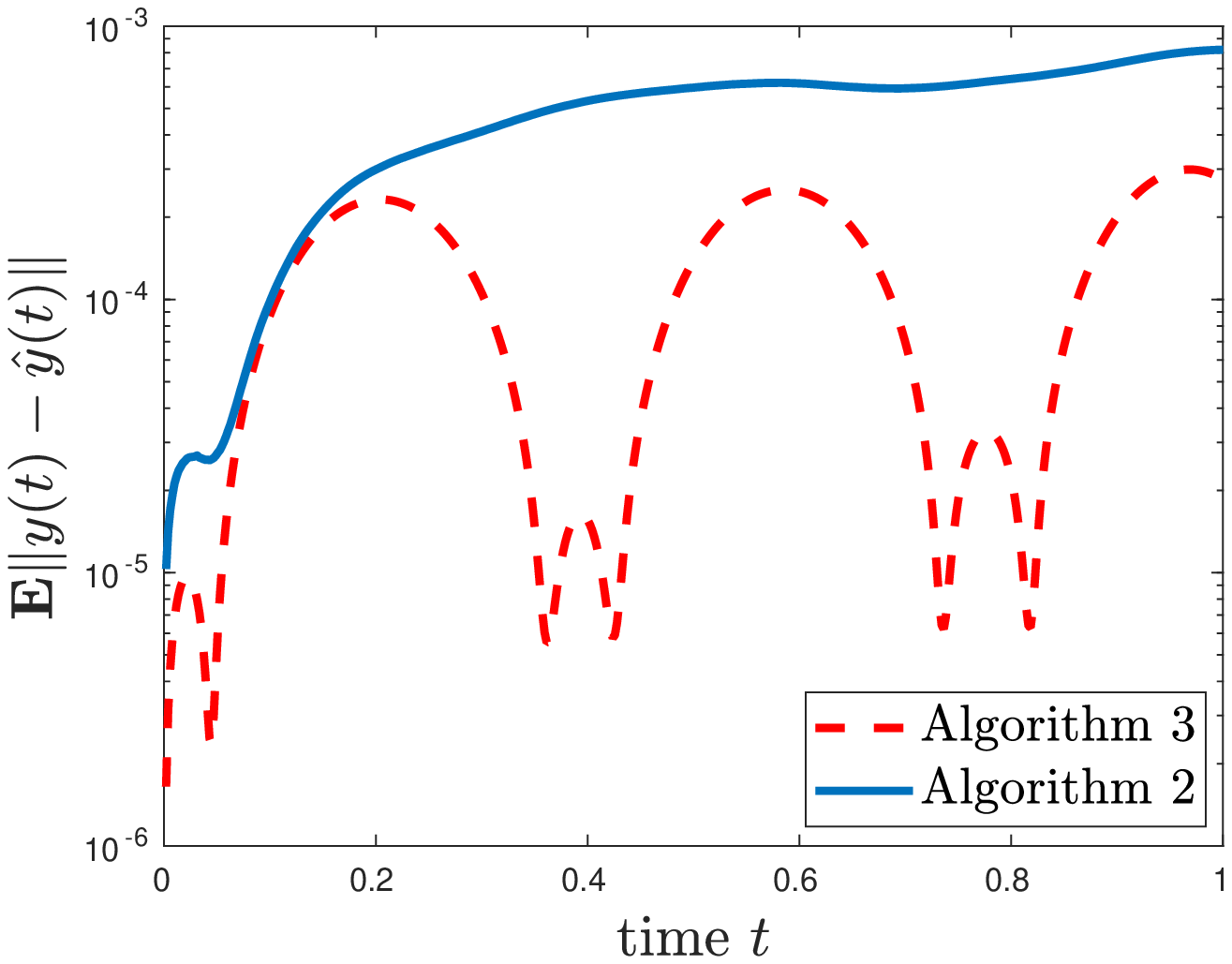}
\caption{Mean errors for the output when using the one step or two step modified linear IRKA.}
\label{fig:error_add_log}
\end{minipage}
\end{figure}
First, we use Theorems \ref{thm_stoch_h2_rep} and \ref{thm_opt_cond} (a) in order to compute the stochastic $\mathcal H_2$-distance $\left\| H - \hat H\right\|^2_{\mathcal L^2(\mathcal W)} = \trace(C P C^T) - \trace(\hat C \hat P {\hat C}^T)$, between a full order model and a ROM for several dimensions of the reduced system, after computing an optimal reduced system of dimensions $r=1,\ldots, 16$.
This allows us to compute $\mathcal{E}_1$, $\mathcal{E}_2$ and $\mathcal{E}_3$ from \eqref{error_est_add} and \eqref{one_step_bound}, which are plotted in Figure \ref{fig:errorbound_add}. 
Clearly, all errors decrease with increasing size of the ROM. However, we observe that the error $\mathcal{E}_2$ decreases much more rapidly than the errors $\mathcal{E}_1$ and $\mathcal{E}_3$, 
which behave similarly. Hence, if we would like to produce a ROM which has an error of at most $2e-02$, we can deduce from Figure \ref{fig:errorbound_add} that we need $r=16$ for the first reduced model
in \eqref{red_system_add_subsystem}, $r = 4$ in the second reduced model in \eqref{red_system_add_subsystem} and $r = 16$ for the reduced model \eqref{red_system_add}. One can see that there are potential savings by using a smaller reduced dimension
and still obtain sufficiently small  errors. For this particular case we computed the ROMs both for Algorithm \ref{algo:IRKA2} with  $r = 16$ and Algorithm \ref{algo:IRKA} with $r=16$ for the first reduced model and 
$r=4$ for the second one. The results are shown in Figure \ref{fig:error_add_log}. We observe that the worst case mean errors of one step and two step modified linear IRKA are of size $2.99e-04$ and $8.20e-04$, respectively.
The errors are of the same order, despite reducing the subsystem corresponding to the stochastic part to a smaller size.

%% file: conclusions.tex
\section{Conclusions}

We have derived optimization based model order reduction methods for stochastic systems. In particular we explained the link between the output error for stochastic systems, both with additive and multiplicative noise, and modified versions of the $\mathcal H_2$-norm for both linear and bilinear deterministic systems. We then developed optimality conditions for minimizing the error bounds computing reduced order models for stochastic systems. This approach revealed that modified versions of iterative rational Krylov methods are in fact natural schemes for reducing large scale stochastic systems with both additive and multiplicative noise.  In addition, we have introduced a splitting method for linear systems with additive noise, where the deterministic and the noise part are treated independently. This is advantageous if one of the systems can be reduced easier than the other. It also allows for a different model order reduction method in one of the systems, which we did not discuss in this paper.

%% file: appendix.tex
\section{Ito product rule}\label{appendixito}
In this section, we formulate an Ito product rule for semimartingales on a complete filtered 
probability space $\left(\Omega, \mathcal F, (\mathcal F_t)_{t\geq 0}, \mathbb P\right)$. Semimartingales $Z$ are stochastic processes that are 
c\`adl\`ag (right-continuous with exiting left limits) that have the representation \begin{align*}
                       Z(t) = M(t) + A(t),\quad t\geq 0,                                                              
                                                                                    \end{align*}
where $M$ is a c\`adl\`ag martingale with respect to $(\mathcal F_t)_{t\geq 0}$ and $A$ is a c\`adl\`ag process with bounded variation.\smallskip

Now, let $Z_1, Z_2$ be scalar semimartingales with jumps $\Delta Z_i(s):=Z_i(s)-Z_i(s-)$ ($i=1, 2$). Then, the Ito product formula 
is given as follows:
\begin{align}\label{profriot}
Z_1(t) Z_2(t)=Z_1(0) Z_2(0)+\int_0^t Z_1(s-)dZ_2(s)+\int_0^t Z_2(s-)dZ_1(s)+[Z_1, Z_2]_t
\end{align}
for $t\geq 0$, see \cite{semimartingalesmichel}. By \cite[Theorem 4.52]{limittheorems}, the compensator 
process $[Z_1, Z_2]$ is given by \begin{align}\label{decomqucov}
[Z_1, Z_2]_t=\left\langle M_1^c, M_2^c\right\rangle_t+\sum_{0\leq s\leq t} \Delta Z_1(s) \Delta Z_2(s)
\end{align}
for $t\geq 0$, where $M_1^c$ and $M_2^c$ are the square integrable continuous martingale parts of $Z_1$ and $Z_2$ (cf. \cite[Theorem 4.18]{limittheorems}).
The process $\left\langle M_1^c, M_2^c\right\rangle$ is the uniquely defined angle bracket process that guarantees that $M_1^c M_2^c-\left\langle M_1^c, M_2^c\right\rangle$ is an $(\mathcal F_t)_{t\geq 0}$- 
martingale, see \cite[Proposition 17.2]{semimartingalesmichel}. As a consequence of (\ref{profriot}), we obtain the following product rule in the vector valued case.
\begin{lem}\label{iotprodformelmatpro}
Let $Y$ be an $\mathbb R^d$-valued and $Z$ be an $\mathbb R^n$-valued semimartingale, then we have\begin{align*}
Y(t) Z^T(t)=Y(0) Z^T(0)+\int_0^t dY(s) Z^T(s-) +\int_0^t Y(s-) dZ^T(s)+\left([Y^{(i)},Z^{(j)}]_t\right)_{{i=1, \ldots, d}\atop {j=1, \ldots, n}}
\end{align*}
for all $t\geq 0$, where $Y^{(i)}$ and $Z^{(j)}$ are the $i$th and $j$th components of $Y$ and $Z$, respectively.
\end{lem}

\section{Proof of Lemma \ref{lemdgl}}\label{proof_prop_fund}

In order to ease notation, we prove the result for $s=0$. Let us assume that $L$ and $\hat L$ have $m$ columns denoted by $l_k$, $\hat l_k$, such that we can decompose $L=\left[l_1, \ldots, l_m\right]$ and
$\hat L=\left[\hat l_1, \ldots, \hat l_m\right]$. we obtain \begin{align}\label{decodaob}
\Phi(t) L \hat L^T\hat {\Phi}^T(t)=\sum_{k=1}^m Y_k(t) Z_k^T(t), 
\end{align}
where we set $Y_k(t) = \Phi(t) l_k$ and $Z_k(t) = \hat \Phi(t) \hat l_k$. We apply Corollary \ref{iotprodformelmatpro} to every summand of \eqref{decodaob}. This yields \begin{align}\label{proditoapplied}
Y_k(t) Z_k^T(t)= l_k \hat l_k^T+\int_0^t d(Y_k(s)) Z_k^T(s-)+\int_0^t Y_k(s-) dZ_k^T(s) +\left(\left[e_{i_1}^T Y_k, e_{i_2}^T Z_k\right]_t\right)_{{i_1=1, \ldots, n}\atop {i_2=1, \ldots, r}},
\end{align}
where $e_{i_1}$ and $e_{i_2}$ are unit vectors of suitable dimension. We determine the expected value of the compensator process of $e_{i_1}^T Y_k$ and $e_{i_2}^T Z_k$. Using \eqref{decomqucov}, it can be seen that this process
only depends on the jumps and the continuous martingale parts of $Y_k$ and $Z_k$. Taking \eqref{funddef} and the corresponding equation for the fundamental solution of the reduced system into account, we see that 
\begin{align*}
 \mathcal M_k(t):=\sum_{i=1}^{m_2}\int_0^t N_i Y_k(s) dM_i(s),\quad \hat{\mathcal M}_k(t):=\sum_{i=1}^{m_2}\int_0^t \hat N_i Z_k(s) dM_i(s)
\end{align*}
are the martingale parts of $Y_k$ and $Z_k$ which furthermore contain all the jumps of these processes. This gives 
\begin{align*}
\left[e_{i_1}^T Y_k, e_{i_2}^T Z_k\right]_t = \left[e_{i_1}^T \mathcal M_k, e_{i_2}^T \hat{\mathcal M}_k\right]_t.
\end{align*}
We apply Corollary \ref{iotprodformelmatpro} to ${\mathcal M}_k\hat{\mathcal M}_k^T$ and obtain 
\begin{align*}
{\mathcal M}_k(t) \hat{\mathcal M}_k^T(t)= \int_0^t d({\mathcal M}_k(s)) \hat{\mathcal M}_k^T(s-)+\int_0^t {\mathcal M}_k(s-) d\hat{\mathcal M}_k^T(s) +\left(\left[e_{i_1}^T {\mathcal M}_k, e_{i_2}^T \hat{\mathcal M}_k\right]_t\right)_{{i_1=1, \ldots, n}\atop {i_2=1, \ldots, r}}.
\end{align*}
Since ${\mathcal M}_k$ and $\hat{\mathcal M}_k$ are mean zero martingales \cite{zabczyk}, the above integrals with respect to these processes have mean zero as well \cite{kuo}. Hence, we have 
\begin{align*}
\mathbb E[{\mathcal M}_k(t) \hat{\mathcal M}_k^T(t)]= \mathbb E \left(\left[e_{i_1}^T {\mathcal M}_k, e_{i_2}^T \hat{\mathcal M}_k\right]_t\right)_{{i_1=1, \ldots, n}\atop {i_2=1, \ldots, r}}=
\mathbb E \left(\left[e_{i_1}^T Y_k, e_{i_2}^T Z_k\right]_t\right)_{{i_1=1, \ldots, n}\atop {i_2=1, \ldots, r}}.
\end{align*}
We apply the expected value to both sides of \eqref{proditoapplied} leading to 
\begin{align*}
\mathbb E [Y_k(t) Z_k^T(t)]= l_k \hat l_k^T+\mathbb E \left[\int_0^t d(Y_k(s)) Z_k^T(s-)\right]+\mathbb E \left[\int_0^t Y_k(s-) dZ_k^T(s)\right] + \mathbb E[{\mathcal M}_k(t) \hat{\mathcal M}_k^T(t)].
\end{align*}
We insert $dY_k$ and $dZ_k$ (given through \eqref{funddef}) into the above equation and exploit that an Ito integral has mean zero, and obtain
\begin{align*}
\mathbb E [Y_k(t) Z_k^T(t)]= l_k \hat l_k^T+\int_0^t A\mathbb E \left[Y_k(s) Z_k^T(s)\right]ds+\int_0^t \mathbb E \left[Y_k(s) Z_k^T(s)\right] \hat A^T ds+ \mathbb E[{\mathcal M}_k(t) \hat{\mathcal M}_k^T(t)].
\end{align*}
Notice that we replaced the left limits by the function values above and hence changed the integrand only on Lebesgue zero sets since the processes have only countably many jumps on bounded time intervals \cite{applebaum}.
 The Ito isometry \cite{zabczyk} now yields \begin{align*}
 \mathbb E[{\mathcal M}_k(t) \hat{\mathcal M}_k^T(t)] = \sum_{i, j = 1}^{m_2}\int_0^t N_i\mathbb E[ Y_k(s)Z_k^T(s)] \hat N_j^T k_{ij} ds.
\end{align*}
Combining this result with \eqref{decodaob} proves the claim of this lemma.

\section{Proof of Theorem \ref{thm:error_opt_cond}}\label{proof_opt_cond}

We only show the result for the optimality condition (c) in \eqref{optcond}.  All the other optimality conditions (a), (b) and (d) are derived similarly. We first reformulate optimality condition (c). Defining 
$\hat\Psi_i := \sum_{j=1}^{m_2} {\hat N}_{j} k_{ij}$ and $\Psi_i := \sum_{j=1}^{m_2} {N}_{j} k_{ij}$ for $i=1,\ldots, m_2$, we have that \begin{align*}
&\hat Q\hat \Psi_i\hat  P = Q_{2} \Psi_i P_{2} \Leftrightarrow (S^{-T}\hat Q)\hat \Psi_i(\hat P S^T) = (S^{-T}Q_{2}) \Psi_i (P_{2} S^T)\\
&\Leftrightarrow \trace\left((S^{-T}\hat Q)\hat \Psi_i(\hat P S^T) e_m e_k^T\right) = \trace\left((S^{-T}Q_{2}) \Psi_i (P_{2} S^T) e_m e_k^T\right)\quad\forall k, m = 1, \ldots, r,
\end{align*}
where $S$ is the factor of the spectral decomposition of $\hat A$. Using the relation between the trace and the vectorization $\vect(\cdot)$ of a matrix as well as the Kronecker product $\otimes$ of two matrices, we obtain an equivalent formulation of the optimality condition (c) in \eqref{optcond}, i.e. $\hat Q\hat \Psi_i\hat  P = Q_{2} \Psi_i P_{2}$ is equivalent to 
\begin{align}\label{equi_opt}
\Leftrightarrow \vect^T(\hat Q S^{-1}) (e_k e_m^T \otimes \hat \Psi_i) \vect(\hat P S^T) = \vect^T(Q_{2}^T S^{-1}) (e_k e_m^T \otimes \Psi_i) \vect(P_{2} S^T)
\end{align}
for all $k, m = 1, \ldots, r$. We now use Algorithm~\ref{algo:MBIRKA} to show that this equality holds.
We multiply \eqref{red_gram} with $S^T$ and \eqref{red_gram_obs} with $S^{-1}$ from the right to find the equations for $\hat P S^T$ and $\hat Q S^{-1}$. Subsequently, we vectorize these equations and get 
\begin{align}\label{firsteq}
   \vect(\hat Q S^{-1}) &= -\hat{\mathcal K}^{-T} \vect(\hat C^T \tilde C)\quad \text{and}\\
  \vect(\hat P S^T) &= -\hat{\mathcal K}^{-1} \vect(\hat B_1 \mathcal W (\tilde B_1\mathcal W)^T),                  
\end{align}
where $\hat {\mathcal K}:= (I\otimes \hat A)+(D\otimes I)+ \sum_{i, j=1}^{m_2} (\tilde N_i\otimes \hat N_j) k_{ij}$ and recalling that $D = S\hat A S^{-1},~\tilde B_1 = S\hat B_1, ~\tilde C = \hat C S^{-1}, ~\tilde N_i = S\hat N_i S^{-1}$.

With ${\mathcal K}:= (I\otimes A)+(D\otimes I)+ \sum_{i, j=1}^{m_2} (\tilde N_i\otimes N_j) k_{ij}$, and the definition of the reduced matrices 
$\hat A = (W^T V)^{-1}W^T AV$, $\hat B_1 = (W^T V)^{-1}W^T B_1$, $\hat C = CV$ and  $\hat N_i = (W^T V)^{-1}W^T N_i V$ in Algorithm~\ref{algo:MBIRKA} we furthermore have that 
\begin{align*}
  - \vect(\hat C^T \tilde C) &= -\vect(V^T C^T \tilde C) = -(I\otimes V^T) \vect(C^T \tilde C) = (I\otimes V^T) \mathcal K^T \vect(W) \\
  &= (I\otimes V^T) \mathcal K^T \vect(W(V^T W)^{-1} V^T W)\\
 & = (I\otimes V^T) \mathcal K^T (I\otimes W(V^T W)^{-1} V^T)  \vect (W) = \hat{\mathcal K}^T (I\otimes V^T)  \vect (W)                       
                          \end{align*}
and 
\begin{align*}
  - \vect(\hat B_1 \mathcal W (\tilde B_1\mathcal W)^T) &= -\vect((W^T V)^{-1} W^T B_1 \mathcal W (\tilde B_1\mathcal W)^T) \\
  &= -(I\otimes (W^T V)^{-1} W^T) \vect(B_1 \mathcal W (\tilde B_1\mathcal W)^T) \\
  &= (I\otimes (W^T V)^{-1} W^T) \mathcal K \vect(V) \\
  &= (I\otimes (W^T V)^{-1} W^T) \mathcal K \vect(V(W^TV)^{-1}W^TV) \\
 & = (I\otimes (W^T V)^{-1} W^T) \mathcal K (I\otimes V(W^TV)^{-1}W^T)  \vect (V)\\
 &= \hat{\mathcal K} (I\otimes (W^T V)^{-1} W^T)  \vect (V)                       
\end{align*}
We insert both results into \eqref{firsteq} and obtain expressions for $ \vect(\hat Q S^{-1})$ and  $\vect(\hat P S^T)$ in terms of the projection matrices from Algorithm~\ref{algo:MBIRKA}:
\begin{align*}
   \vect(\hat Q S^{-1}) = (I\otimes V^T)  \vect (W) \quad \text{and}\quad  \vect(\hat P S^T) = (I\otimes (W^T V)^{-1} W^T)  \vect (V).                  
\end{align*}
Hence, the left hand side of the optimality condition \eqref{equi_opt} can be written as
\begin{align*}
&\vect^T(\hat Q S^{-1}) (e_k e_m^T \otimes \hat \Psi_i) \vect(\hat P S^T) \\
&= \vect^T(W) (I\otimes V) (e_k e_m^T \otimes \hat \Psi_i) (I\otimes (W^T V)^{-1} W^T)  \vect (V)\\
&= \vect^T(W) (I\otimes V(W^T V)^{-1} W^T)) (e_k e_m^T \otimes \Psi_i) (I\otimes V(W^T V)^{-1} W^T)  \vect (V)\\
&= \vect^T(W)  (e_k e_m^T \otimes \Psi_i)  \vect (V),
\end{align*}
where we have used properties of the Kronecker product again. It remains to show that  
$P_{2} S^T = V$ and $Q_{2}^T S^{-1} = W$ for the optimality condition to hold. This is obtained by multiplying \eqref{mixed_gram} with $S^T$ from the right and \eqref{mixed_gram_obs} with $S^{-T}$ from the left. Hence \eqref{equi_opt} holds which concludes the proof.